\documentclass[11pt,a4paper]{article}
\usepackage[utf8]{inputenc}
\usepackage[T1]{fontenc}

\usepackage{amsmath}
\usepackage{amsthm}
\usepackage{amssymb}
\usepackage{amsopn}
\usepackage{enumitem}
\usepackage{color}
\usepackage{mathtools}
\usepackage{hyperref}
\usepackage{doi}

\usepackage[margin=3cm]{geometry}

\theoremstyle{plain}
\newtheorem{theorem}{Theorem}[section]
\newtheorem{proposition}[theorem]{Proposition}
\newtheorem{corollary}[theorem]{Corollary}
\newtheorem{lemma}[theorem]{Lemma}

\theoremstyle{definition}
\newtheorem{definition}[theorem]{Definition}
\newtheorem{example}[theorem]{Example}

\theoremstyle{remark}
\newtheorem{remark}[theorem]{Remark}

\DeclareMathOperator{\Id}{Id}

\DeclareMathOperator{\sym}{sym}

\DeclareMathOperator{\skewo}{skew}

\DeclareMathOperator{\dist}{dist}

\DeclareMathOperator*{\esssup}{ess\,sup}

\def\Xint#1{\mathchoice
   {\XXint\displaystyle\textstyle{#1}}%
   {\XXint\textstyle\scriptstyle{#1}}%
   {\XXint\scriptstyle\scriptscriptstyle{#1}}%
   {\XXint\scriptscriptstyle\scriptscriptstyle{#1}}%
   \!\int}
\def\XXint#1#2#3{{\setbox0=\hbox{$#1{#2#3}{\int}$}
     \vcenter{\hbox{$#2#3$}}\kern-.5\wd0}}

\def\dashint{\Xint-}

\newcommand\matr[1]{\begin{pmatrix} #1 \end{pmatrix}}  
\newcommand\abs[1]{\left|#1\right|} 
\newcommand\norm[1]{\left \Vert #1\right \Vert} 

\newcommand{\Z}{\mathbb{Z}}

\newcommand{\N}{\mathbb{N}}
\newcommand{\R}{\mathbb{R}}

\newcommand{\eps}{\varepsilon}
\newcommand{\dtilde}[1]{\tilde{\raisebox{0pt}[0.85\height]{$\tilde{#1}$}}}

\newcommand{\cell}{\normalfont\text{cell}} 
\newcommand{\surf}{\normalfont\text{surf}} 
\newcommand{\rel}{\normalfont\text{rel}}

\newcommand\doublebar[1]{{\bar{\bar{#1}}}} 
\newcommand{\Ox}{\Omega}
\newcommand{\Oxc}{\overline{\Omega}}

\newcommand{\Lambra}{\Lambda}

\makeatletter
\renewcommand*\env@matrix[1][*\c@MaxMatrixCols c]{%
  \hskip -\arraycolsep
  \let\@ifnextchar\new@ifnextchar
  \array{#1}}
\makeatother

\begin{document}
\begin{center}
\begin{LARGE}
A derivation of the time dependent von Kármán equations \\[2mm]  from atomistic models
\end{LARGE}
\\[0.7cm]
\begin{large}
David Buchberger\footnote{Universität Augsburg, Germany, {\tt buchbd@posteo.de}} and Bernd Schmidt\footnote{Universität Augsburg, Germany, {\tt bernd.schmidt@math.uni-augsburg.de}}
\end{large}
\\[0.5cm]
\today
\\[1cm]
\end{center}

\begin{abstract}
We derive the time-dependent \emph{von Kármán plate equations} from three dimensional, purely atomistic particle models. 
In particular, we prove that a thin structure of interacting particles whose dynamics is governed by Newton's laws of motion is effectively described by the von Kármán equations in the limit of vanishing interatomic distance $\eps$ and vanishing plate thickness $h$. 
While the classical plate equations are obtained for $\eps \ll h \ll 1$, we find new plate equations for finitely many layers in the ultrathin case $\eps \sim h$. 
\end{abstract}

\section{Introduction}

    The aim of this contribution is to provide a justification of the time-dependent von Kármán plate equations as an effective model for a thin structure of classically interacting atomistic particles. 

    Von Kármán's plate equations play a central role in the continuum mechanical description of plate structures. Formulated by Theodore von Kármán in the first half of the 20th century \cite{vonKarman}, they continue to find a wide range of applications in the engineering sciences. To justify these equations from basic continuum mechanical models is thus a major challenge in mathematical elasticity theory, cf.\ e.g.\ \cite{Ciarlet:97,Antman:95} for classical approaches. Only comparably recently a full justification from 3d nonlinear elasticity theory in the time dependent setting was achieved by Abels, Mora and Müller in \cite{AbelsMoraMueller_09,AbelsMoraMueller_11}. 

    First rigorous results deriving lower dimensional effective theories as a $\Gamma$-limit from 3d elasticity in the static setting were obtained in \cite{AcerbiButtazzoPercivale:91,AnzellottiBaldoPercivale:94,LeDretRaoult1995}. Most notably, \cite{FrieseckeJamesMueller_02,FrieseckeJamesMueller_06} have triggered extensive further development in this area. In particular, the energy functional for von Kármán plate theory has been obtained in \cite{FrieseckeJamesMueller_06}. In contrast, convergence results beyond those for (almost) minimizers are far less well understood. In \cite{MoraMuellerSchultz:07,MoraMueller:08,MuellerPakzad_08} the convergence of equilibria for rods and von Kármán plates plates, respectively, which might not be absolute minimizers was established under additional linear growth assumptions on the derivative of the stored energy functional. These assumptions could be weakened to a physically more realistic growth assumption in \cite{MoraScardia_09, DavoliMora_12}. In a quasi-static regime, versions of the von Kármán plate equations have been derived for plastic and viscoelastic materials, cf.\ \cite{Davoli:14,FriedrichKruzik:20}. Except for the aforementioned contribution \cite{AbelsMoraMueller_09}, little seems to be known on dimensionally reduced time-dependent systems that take proper account of inertial effects except for the recent papers \cite{AbelsAmeismeier_22,AbelsAmeismeier_23}. 

    For very small objects on the nanoscale in particular, the natural question arises whether atomistic effects should be taken into account in order to appropriately describe specimens which only consist of a few atomic layers in their transversal direction. This leads to considering interacting particle models and to deriving effective theories directly from their atomistic interaction potentials. Seminal work on such discrete-to-continuum limits were obtained in \cite{AlicandroCicalese:04,BraidesSolciVitali:07,BraunSchmidt:13}. In combination with a simultaneous dimension reduction, such passages from discrete to continuum have been achieved in \cite{FrieseckeJames:00,Schmidt:06,Schmidt:08,AlicandroBraidesCicalese:08,SchmidtZeman2023,SchmidtZeman2023b} and, for the von Kármán setting in \cite{BraunSchmidt2022}. Here, in contrast to the continuum derivation one considers atomistic interaction energies of the type 
    \[ E(y) = \sum_{\Lambda_{n,S}} W_{\text{cell}}(\bar{\nabla} y(x)) + \text{surface terms}, \]
    where $\Lambda_{n,S}$ is a finite (cubic) lattice occupying the region of a macroscopic plate of height $h_n$ with interatomic distance $\varepsilon_n$ and $\bar{\nabla}$ is a finite difference operator capturing nearest and next-to-nearest atom interactions.
    Such an atomistic approach does indeed describe both thin films (with $0 < \varepsilon_n \ll h_n \ll1$) and even `ultrathin' films (with $0 < \varepsilon_n \sim h_n \ll1$) adequately in the simultaneous limit $\varepsilon_n, h_n \searrow 0$. 
    In the latter case, the authors obtained a new von Kármán plate theory which also features surface contributions that cannot be neglected for such thin objects. For analogous results on Kirchhoff plates and rods, respectively, we refer to \cite{Schmidt:06,SchmidtZeman2023}.

    In the present paper we start with time dependent atomistic deformations that obey Newton's second law of motion. We show that, under suitable assumptions on the interaction potential, rescaled displacements of these maps converge to (weak) solutions of the time-dependent von Kármán equations as both the interatomic distance and the plate height tend to zero. We capture both cases, thin and ultrathin films. To the best of our know\-ledge, we thus provide the first combined discrete-to-continuum passage with simultaneous dimension reduction in the time dependent setting of elastodynamics. Our main result is Theorem~\ref{thm: main-td}. Its general proof strategy is motivated by the continuum results in \cite{AbelsMoraMueller_09}. Yet, our discrete setup leads to severe technical difficulties as we need to provide accurate estimates on finite differences of particle positions and, in particular, obtain control of a suitable atomistic strain tensor in space and time. For thin films, the challenge is to show that for solutions of Newton's equations all discreteness effects fade away and we obtain the classical von Kármán equations in the limit. For ultrathin films the situation is even more involved. Here the plate will not homogenize in the small $x_3$-direction and we need to exactly identify novel contributions to the limiting equations that appropriately capture those residual discreteness effects as well as surface contributions which are of the same order as the bulk terms. As corollaries we obtain a global in time existence result for the limiting equations, see Remark~\ref{rmk:global-existence}, and a corresponding discrete-to-continuum convergence result in the stationary setting, see Theorem~\ref{thm: stationary}.

	\subsection{Notation}
	For $x \in \R^3$ we write $x = (x',x_3)$ with $x' = (x_1, x_2) \in \R^2$. For any matrix $A \in \R^{3 \times 8}$ we define by $A^{(1)} := (A_{\cdot 1}, \ldots , A_{\cdot 4}) \in \R^{3 \times 4}$ and $A^{(2)} := (A_{\cdot 5}, \ldots , A_{\cdot 8}) \in \R^{3 \times 4}$ the matrices consisting of the first four respectively the last four columns of $A$. If $A \in \R^{2 \times 2}$, we write 
		\begin{align*}
			\hat{A} = \matr{A & 0 \\ 0 & 0} 
		\end{align*}
	for its trivial extension to a $3 \times 3$ matrix.
	The matrix of the in-plane-derivatives of a map $y \colon \R^3 \to \R^k$ will be denoted by $\nabla'y = (\partial_1 y, \partial_2 y) \in \R^{k \times 2}$.


\section{Models and results}\label{section:Models-and-Results}


	\subsection{Kinematics of the atomic model}
		Let $S \subset \R^2$ be an open, bounded and connected Lipschitz domain. Let $\varepsilon_n, h_n > 0$ such that $\varepsilon_n, h_n \to 0$, where $\varepsilon_n$ corresponds to the interatomic distance and $h_n$ to the height of the plate. By $\nu_n \in \N$ we denote the number of layers in the $x_3$-direction, i.e. we have
		\[ h_n = (\nu_n - 1) \varepsilon_n. \]
		$\Lambda_{n,S} = (\overline{S} \times [0, h_n]) \cap \varepsilon_n \Z^3$ is the set of atomic positions in the reference configuration. We also set $\Lambra_n = (\R^2 \times [0, h_n]) \cap \varepsilon_n \Z^3$. An atomistic (or `discrete' or `lattice') deformation $w$ is a mapping
		\[w \colon \Lambda_{n,S} \to \R^3.\]
		For simplicity we are modeling fully clamped plates.
		To this end, we will tacitly extend deformations to mappings $w \colon \Lambra_n \to \R^3$ by setting $w(x) = (x_1,x_2,x_3 - \frac{h_n}{2})$ for $x \in \Lambra_n \setminus \Lambda_{n,S}$. 
		
		We define the \emph{dual grids} $\Lambda_{n,S}^\prime$ and $\Lambra_n^\prime$ as the set $x$ of midpoints of `inner', respectively all, lattice cells $x + [- \frac{\varepsilon_n}{2},\frac{\varepsilon_n}{2}]^3$ such that $x + \{- \frac{\varepsilon_n}{2},\frac{\varepsilon_n}{2}\}^3$ is contained in $\Lambda_{n,S}$, respectively, $\Lambra_n^\prime$. 
        For $x \in \Lambra_n^\prime$ we set
		\[Q_n(x) = x + \left(-\frac{\varepsilon_n}{2},\frac{\varepsilon_n}{2}\right)^3.\]
				
		Let $z^1, \ldots, z^8$ be the corners of the unit cube centered at $0$ such that 
		\[Z = \left( z^1, \ldots, z^8 \right) = \frac{1}{2} 
		\matr{-1 & 1 & 1 & -1 & -1 & 1 & 1 &-1 \\
			-1 & -1 &1 &1 &-1 &-1 &1 &1 \\
			-1 &-1 &-1 &-1 &1 &1 &1 &1
		}.
		\]
		For an (extended) atomistic deformation we define the associated discrete gradient
		\[\bar{\nabla} w(x) = \left( {\bar\partial_1} w(x), \ldots, {\bar\partial_8}w(x) \right) \in \R^{3 \times 8}, \quad x \in \Lambra_n^\prime, \label{discrete-gradient}\]
		where
		\[{\bar\partial_i} w(x) = \frac{1}{\varepsilon_n} \biggl( w(x + \varepsilon_n z^i ) - \frac{1}{8} \sum_{j=1}^{8} w(x + \varepsilon_n z^j) \biggr). \label{discrete-partial-derivative}\]
		This defines a linear operator $\bar{\nabla} \colon \{ \varphi \colon \Lambra_n \to \R^3 \} \to \{ \phi: \Lambra_n^\prime \to \R^{3 \times 8} \}$. When equipped with their natural inner products, it admits a formal adjoint $\bar{\nabla}^\ast \colon  \{ \phi: \Lambra_n^\prime \to \R^{3 \times 8} \} \to \{ \varphi \colon \Lambra_n \to \R^3 \}$ given by the identity
		\begin{equation}
			\sum_{\omega \in \Lambra_n^\prime} F(\omega) \colon \left( \bar{\nabla} \varphi \right)(\omega) 
			= \sum_{x \in \Lambra_n} \left( \bar{\nabla}^\ast F \right) (x) \cdot \varphi(x), \label{discrete-adjoint}
		\end{equation}
		where $F: \Lambra_n^\prime \to \R^{3 \times 8}$ and $\varphi: \Lambra_n \to \R^3$ has compact support. For later purposes we notice that $\bar{\nabla}^\ast F$ is given explicitly by the following formula (which can be found by inserting $\varphi(x) = b \delta_{x_0}$ for generic $b \in \R^3$ and $x_0 \in \Lambra_n)$ in \eqref{discrete-adjoint}): 
		\begin{align}
		\bar{\nabla}^\ast F(x) 
		= \sum_i \frac{1}{\varepsilon_n} \bigg( F_{\cdot i} (x - \varepsilon_n z^i) - \frac{1}{8} \sum_{j=1}^{8} F_{\cdot j}(x - \varepsilon_n z^j) \bigg), \label{discrete-adjoint-explicit}
		\end{align}
        where the sum over $i$ runs from $1$ to $8$ in the bulk, i.e., $x_3 \in \{\frac{1}{\nu_n-1}, \ldots, \frac{\nu_n-2}{\nu_n-1}\}$, from $5$ to $8$ in the top layer $x_3 = 1$ and from $1$ to $4$ in the bottom layer $x_3 = 0$. $\bar{\nabla}^\ast$ can be interpreted as a discrete divergence with respect to the second index in $F$. 
		
		We also notice that $\bar{\partial}_i$ obeys the following product rule: 
		\begin{align}\label{eq: bar-product-rule} 
		 {\bar\partial_i} (fg)(x) 
		 &= {\bar\partial_i} f(x) g(x + \varepsilon_n z^i)  + \frac{1}{8} \sum_{j=1}^{8} f(x + \varepsilon_n z^j) \cdot \big( {\bar\partial_i} g(x) - {\bar\partial_j} g(x) \big). 
		\end{align}
				

		It will be convenient to work on a fixed macroscopic domain. To this end, we let  
		$H_n : \R^3 \to \R^3$ be the linear mapping given by $H(x) = (x',h_nx_3)$
        and introduce the rescaled lattices 
        \begin{align*}
            \tilde{\Lambda}_{n,S} = H_n^{-1} \Lambda_{n,S}, \qquad 
            \tilde{\Lambra}_n = H_n^{-1} \Lambra_n, \qquad 
            \tilde{\Lambda}^\prime_{n,S} = H_n^{-1} \Lambda^\prime_{n,S}, \qquad 
            \tilde{\Lambra}^\prime_n = H_n^{-1} \Lambra^\prime_n.
        \end{align*}
        An (extended) deformation $w: \Lambra_n \to \R^3$ will be rescaled to the deformation $y: \tilde{\Lambra}_n \to \R^3$ via $y(x) = w(H_n x)$. The rescaled discrete gradients are given by
		\begin{equation}
			\bar\partial_i^n y(x)  = \frac{1}{\varepsilon_n} \biggl[ y \biggl (x' + \varepsilon_n \left( z^i \right)^\prime, x_3 + \frac{\varepsilon_n}{h_n} z_3^i \biggr) - 	\frac{1}{8} \sum_{j=1}^{8} y \biggl(x' + \varepsilon_n \left( z^j \right)^\prime, x_3 + \frac{\varepsilon_n}{h_n} z_3^j \biggr)\biggr]. \label{discrete-partial-derivative-rescaled}
		\end{equation}
	\subsection{The energy and the quadratic forms}
	We assume that the atomic interaction energy for a deformation can be written as
	\[  E_{\text{atom}}(w) 
	= \sum_{x \in \Lambda_n^{\prime}} W \left(h_n^{-1}x_3, \bar{\nabla} w(x) \right).\] 
	As we will choose $W$ below such that $W(x_3,Z)=0$ and $w$ is clamped at the lateral boundary, in fact only cells that intersect $S \times (0,h_n)$ contribute to $E_{\text{atom}}(w)$.
	We further assume that $W$ is given by a homogeneous cell energy $W_{\cell}: \R^{3\times 8} \to [0,\infty)$ together with homogeneous surface term $W_{\surf}:\R^{3\times 4} \to [0,\infty)$. More precisely,
	\begin{align}
		W(x_3,A) =
		\begin{cases}
			W_{\cell}(A) & \quad \text{if} \quad x_3 \in  (\frac{1}{\nu_n-1}, \frac{\nu_n-2}{\nu_n-1}), \\
			W_{\cell}(A) + W_{\surf}(A^{(2)}) & \quad \text{if} \quad \nu_n \geq 3, \ x_3 \in (\frac{\nu_n-2}{\nu_n-1}, 1), \\
			W_{\cell}(A) + W_{\surf}(A^{(1)}) & \quad \text{if} \quad \nu_n \geq 3, \ x_3 \in (0, \frac{1}{\nu_n-1}), \\
			W_{\cell}(A) + \sum_{i=1}^{2} W_{\surf}(A^{(i)}) & \quad \text{if} \quad \nu_n = 2, \ x_3 \in (0,1).				
		\end{cases} \label{homogeneous-energy}
	\end{align}
	
	We assume that $W(x_3, \cdot)$ is frame indifferent satisfying 
	\begin{align*}
		W_{\cell} (RA+ (c, \ldots, c)) 
		= W_{\cell}(A) 
		\quad \text{and} \quad 
		W_{\surf} (RA^{(1)}+ (c, c, c, c)) 
		= W_{\surf}(A^{(1)}) 
    \end{align*}
	for all $R \in SO(3)$, $c \in \R^3$ and $A \in \R^{3 \times 8}$, minimal at the identity, i.e., 
	\begin{align*}
        W_{\cell}(Z) 
        = W_{\surf} (Z^{(1)}) 
        = 0, 
	\end{align*}
    (and thus even $W_{\cell}(RZ) = W_{\surf}(RZ^{(1)}) = 0$ for every $ R \in SO(3)$) and that 
    \begin{align*} 
    W_{\cell} \text{ and } W_{\surf} \text{ are } C^2        \text{ in a neighborhood of } Z, \text{ respectively, } Z^{(1)}. 
    \end{align*}
    (We only consider $Z^{(1)}$ in $W_{\surf}$ since $Z^{(2)} = Z^{(1)} + (e_3,e_3,e_3,e_3)$.) The last property will allow for a Taylor expansion around $Z$, respectively, $Z^{(1)}$.  
	Moreover we assume that there is a $c_0 > 0$ such that
	\begin{equation}
		W_{\cell}(A) \geq c_0 \ \dist^2(A, SO(3)Z) \label{cellass5} 
	\end{equation}
	whenever $\sum_{i=1}^{8} A_{\cdot i} = 0$.
	Lastly, we make the technical assumption that $W_{\cell}$ and $W_{\surf}$ are $C^1$ and the derivatives satisfy a linear growth condition 
	\begin{align}\label{derivative-est}
		\vert D W_{\cell}(A) \vert + \vert DW_{\surf}(A^{(1)}) \vert \leq C(1 + \vert A \vert) \quad \text{for every} \ A \in \R^{3 \times 8}. 
	\end{align}
	
	For the rescaled deformation $y$ of $w$ we also write 
	\[E_{\text{atom}}(y)
	= \sum_{x \in \tilde{\Lambda}_{n,S}^{\prime}} W(x, \bar{\nabla}_n y(x) ). \]
	
	In addition to the atomic interaction energy we consider an energy contribution from body forces $g_n: \tilde{\Lambda}_{n,S} \to \R^3$ of the form $g_n(x) = h_n^3 g(x')$ for $g$ sufficiently regular. We extend them to $\tilde{\Lambra}_n$ by setting $g_n(x) = 0$ for $\in \tilde{\Lambra}_n \setminus \tilde{\Lambda}_{n,S}$. The overall energy per unit volume is then given by
	\begin{equation} 
		E_n (y) = \frac{\varepsilon_n^3}{h_n} \biggl[ \sum_{x \in \tilde{\Lambra}^\prime_n} W(x,\bar{\nabla}_n y(x)) 
		+ \sum_{x \in \tilde{\Lambra}_n} g_n(x') \cdot y(x)  \biggr]. \label{discrete-energy}
	\end{equation}

	Computing the first variation with test functions $\varphi : \tilde{\Lambra}_n \to \R^3$ that vanish on $\tilde{\Lambra}_n \setminus \tilde{\Lambda}_{n,S}$ leads to 
	\begin{align*}
	\delta E_n (y, \varphi) = \frac{\varepsilon_n^3}{h_n} \biggl[ \sum_{x \in \tilde{\Lambra}^\prime_n} D_A W(x,\bar{\nabla}_n y(x)) : \bar{\nabla}_n \varphi(x)
		+ \sum_{x \in \tilde{\Lambra}_n} g_n(x') \cdot \varphi(x)  \biggr]. 
	\end{align*}
	Thus, under a deformation $y$, the force $F_n^y(x)$ that is exerted on the atom with reference point $x \in \tilde{\Lambda}_{n,S}$ is 
	\begin{equation} 
		F_n^y(x) = \frac{\varepsilon_n^3}{h_n} \bigl[ \bar{\nabla}_n^\ast (D_A W(\cdot,\bar{\nabla}_n y(\cdot)) )(x)
		+ g_n(x') \bigr]. \label{discrete-force}
	\end{equation}
	
	As we are ultimately interested in small displacements, we next introduce the quadratic forms of infinitesimal elasticity and their relaxations. These forms will also appear in our main theorem.
	For $A \in \R^{3 \times 8}$ let
	\[Q_{\cell}(A) = D^2W_{\cell}(Z)[A,A].\]
	
	We associate a relaxed quadratic form on $\R^{3 \times 8}$ given by
	\[
	Q_{\cell}^{\rel}(A) \coloneqq \min_{b \in \R^3} Q_{\cell}(A + (b \otimes e_3)Z).
	\]
	Since, as a consequence of \eqref{cellass5}, $Q_{\cell}$ is positive definite on $\left(\R^3 \otimes e_3\right)Z$,  for every $A \in \R^{3 \times 8}$ there is a unique $b(A) \in \R^3$ such that
	\[Q_{\cell}^{\rel} \left( A \right) = Q_{\cell}\left(A + \left(b(A) \otimes e_3\right)Z\right). \]
	This $b(A)$ is characterized by
	\begin{align}\label{eq:c-min-char}
	DQ_{\cell} \left( A + \left(b(A) \otimes e_3\right)Z \right) \colon \left(c \otimes e_3\right)Z = 0 
	\end{align}
	for every $c \in \R^3$, i.e. $DQ_{\cell} \left(A + \left(b(A) \otimes e_3\right) Z \right) \perp \left(\R^3 \otimes e_3\right)Z$. Moreover, the map $A \mapsto b(A)$ is linear. 
	Analogously let 
	\[ Q_{\surf}(A) = D^2W_{\surf}(Z^{(1)})[A,A] \]
	for $A \in \R^{3 \times 4}$.

	\subsection{Equations of motion}
	By $I_T$ we denote the interval $[0,T]$ for $T \in (0,\infty)$. If $T = \infty$ we set $I_T = [0,\infty)$. We assume that the dynamics of the atomistic particle system obeys Newton's second law of motion. In accordance with the small film height and the small body forces we also rescale time and obtain the following high-dimensional system of ordinary differential equations of motion
	\begin{equation}
		h_n^2 \partial_t^2 y_n(t,x) = -\bar{\nabla}_n^\ast \left( D_A W \left(\cdot, \bar{\nabla}_n y_n(t,\cdot)\right) \right)(x) + h_n^3 g(t,x') e_3,\label{strong-solution}
	\end{equation}
	for an atomistic deformation $y_n \colon I_T \times \tilde{\Lambda}_{n,S} \to \R^3$ with $y_n(\cdot,x) \in H^{2}_{\mathrm{loc}}(I_T)$ for all $x \in \tilde{\Lambda}_{n,S}$. 
	After extension to $\tilde{\Lambra}_n$, the weak form of this equation (which in the discrete setting is even equivalent) reads as follows:
	Testing with $\varphi \colon I_T \times \tilde{\Lambra}_n \to \R^3$ such that $\varphi(\cdot, x) \in H_0^1(I_T)$ for every $x \in \tilde{\Lambda}_{n,S}$ and $\varphi(\cdot,x) = 0$ whenever $x \in \tilde{\Lambra}_n \setminus \tilde{\Lambda}_{n,S}$ yields
    \begin{align}
		0 &= 
		h_n^2 \int_0^{T'} \frac{\varepsilon_n^3}{h_n} \sum_{x \in \tilde{\Lambra}_n} \partial_t y_n(t,x) \cdot \partial_t \varphi(t,x) \ dt \nonumber \\
        &\qquad - \int_0^{T'} \frac{\varepsilon_n^3}{h_n} \sum_{x \in \tilde{\Lambra}^\prime_n} D_A W\left( x, \bar{\nabla}_n y_n(t,x) \right) \colon \bar{\nabla}_n \varphi(t,x) \ dt \nonumber  \\
		&\qquad+ \int_0^{T'} \frac{\varepsilon_n^3}{h_n} \sum_{x \in \tilde{\Lambra}_n}  h_n^3 g(t,x') \varphi_3(t,x) \ dt. \label{dyn_weak_sltn}
    \end{align}

	In the continuum setting we are led to considering the  time-dependent von Kármán equations. In the following we will work with the spaces $L^p_{\mathrm{loc}}(I_T;X)$, $p \in [1,\infty]$, for different Banach spaces $X$. In particular if $T < \infty$ the space $L^p_{\mathrm{loc}}(I_T;X)$ agrees with the space $L^p(I_T;X)$.
	We have to distinguish between $\nu_n \to \infty$ and $\nu_n \equiv \nu \in \N$. To make sense of the upcoming definitions we assume that the force term is sufficiently regular to have suitable integrability as well as point evaluation. This is certainly the case for $g \in L^2_\mathrm{loc}(I_T; W^{1,\infty}(S))$.
	In Section~\ref{subsec:Interpolation-schemes} below we will introduce a particularly convenient interpolation scheme which associates to any (rescaled) deformation $y_n : I_T \times \tilde{\Lambra}_n \to \R^3$ an interpolation $\tilde{y}_n \in L^\infty_{\mathrm{loc}}(I_T; H^1_{\mathrm{loc}}(\R^2 \times (0,1); \R^3))$. We then consider the rescaled in plane and out-of-plane components 
	\begin{align}
		u_n(t,x') 
		&\coloneqq h_n^{-2} \int_0^1 \big( {\tilde{y}_n^\prime}(t,x',x_3) - x' \big) \ dx_3, \label{eq:un-def} \\
		v_n(t,x') 
		&\coloneqq h_n^{-1} \int_0^1 \big(\tilde{y}_n (t,x',x_3) \big)_3 \ dx_3. \label{eq:vn-def} 
	\end{align}

	For given $u \in H^1_0(S;\R^2)$ and $v \in H^2_0(S)$ we consider the first and second order strain terms 
	\begin{align*}
	 G_1^{\sym} 
	 = \sym \nabla' u + \frac{1}{2} \nabla' v \otimes \nabla' v, \quad 
	 G_2 
	 = - \nabla'^2 v \quad(\in \R^{2 \times 2}) 
	\end{align*}
	With the help of the $3 \times 8$ matrices 
	\begin{align*}
	 Z_- = (-Z^{(1)}, Z^{(2)})
	 \quad\text{and}\quad 
	 M = \frac{1}{2} e_3 \otimes (+1,-1,+1,-1,+1,-1,+1,-1)  
	\end{align*}
	we then define `von Kármán superposition operators' 
	\begin{align*} 
	 \mathcal{B}^\nu_{{\rm vK},1} 
	 &: H^1_0(S;\R^2) \times H^2_0(S) \times H^1_0(S) \to L^1(S), \\ 
	 \mathcal{B}^\nu_{{\rm vK},2} 
	 &: H^1_0(S;\R^2) \times H^2_0(S) \times H^2_0(S;\R^2) \to L^1(S)	 
	\end{align*}
    as follows.  
	In case $\nu_n \to \infty =: \nu$ we consider 
    \begin{align*}
	 \mathcal{B}^\infty_{{\rm vK},1}(u, v, \phi)
	 &:= \frac{1}{2}DQ_{\cell}^{\rel} ( \hat{G}_1^{\sym} Z) \colon (\nabla' v \otimes \nabla' \phi)\widehat{\phantom{i}} \, Z + \frac{1}{24} DQ_{\cell}^{\rel} \left( \hat{G}_2 Z \right) \colon (\nabla'^2 \phi )\widehat{\phantom{i}} \, Z \\  
     \shortintertext{and}
     \mathcal{B}^\infty_{{\rm vK},2}(u, v, \Psi)
	 &:= \frac{1}{2}DQ_{\cell}^{\rel} (\hat{G}_1^{\sym} Z) \colon (\nabla' \Psi)\widehat{\phantom{i}} \, Z \ , 
	\end{align*}
	whereas for $\nu_n \equiv \nu \in \{2,3,\ldots\}$ we set 
	\begin{align*}
	 \mathcal{B}^\nu_{{\rm vK},1}(u, v, \phi)
	 &:=  \frac{1}{2} DQ_{\cell}^{\rel} \left(\hat{G}_1^{\sym} Z + \frac{1}{2(\nu-1)} ( \hat{G}_2 Z_- + \partial_{12} v \, M) \right) \nonumber \\ 
     &\qquad \qquad \qquad \colon 
	 \left( (\nabla' v \otimes \nabla' \phi)\widehat{\phantom{i}} \, Z 
	 - \frac{1}{2(\nu - 1)}( (\nabla'^2 \phi)\widehat{\phantom{i}} \, Z_-
	 + \partial_{12} \phi \, M) \right) \nonumber \\
	 &\qquad - \frac{\nu (\nu - 2)}{24(\nu - 1)^2} DQ_{\cell}^{\rel} (\hat{G}_2 Z) \colon (\nabla'^2 \phi )\widehat{\phantom{i}} \, Z  \nonumber \\
	 &\qquad + \frac{1}{\nu - 1} DQ_{\surf} \left( \hat{G}_1^{\sym} Z^{(1)} + \frac{1}{2(\nu-1)} \partial_{12} v \, M^{(1)} \right) \nonumber \\ 
     &\qquad \qquad \qquad \qquad \colon 
	 \left( (\nabla' v \otimes \nabla' \phi)\widehat{\phantom{i}} \, Z^{(1)} 
	 + \frac{1}{2(\nu - 1)} \partial_{12} \phi \, M^{(1)} \right) \nonumber \\
	 &\qquad - \frac{1}{4(\nu - 1)} DQ_{\surf} ( \hat{G}_2 Z^{(1)} ) \colon (\nabla'^2\phi)\widehat{\phantom{i}} Z^{(1)} \\  
     \shortintertext{and}
     \mathcal{B}^\nu_{{\rm vK},2}(u, v, \Psi)
	 &:= \frac{1}{2} DQ_{\cell}^{\rel} \left(\hat{G}_1^{\sym} Z + \frac{1}{2(\nu-1)} ( \hat{G}_2 Z_- +\partial_{12} v \, M) \right) \colon 
     (\nabla' \Psi)\widehat{\phantom{i}} \, Z \nonumber \\
	 &\qquad+ \frac{1}{\nu - 1}  DQ_{\surf} \left( \hat{G}_1^{\sym} Z^{(1)} + \frac{1}{2(\nu-1)} \partial_{12} v \, M^{(1)} \right) \colon (\nabla' \Psi)\widehat{\phantom{i}} \, Z^{(1)}. 
	\end{align*}
	
    \begin{definition}
		Let $\nu \in \{2,3,\ldots,\infty\}$. We say a pair $(u,v)$ is a weak solution to the dynamic von Kármán equations if $u \in L^\infty_{\mathrm{loc}}(I_T; H^1_0(S;\R^2) )$, $v \in L^\infty_{\mathrm{loc}}(I_T;H^2_0(S)) \cap W^{1,\infty}_{\mathrm{loc}}(I_T;L^2(S)) $ and for every $T' \in I_T$ the following two equations are satisfied:
		\begin{align}
		\int_{0}^{T'} \int_S \left( 
		\partial_t v \, \partial_t \phi   
		- \mathcal{B}^\nu_{{\rm vK},1}(u, v, \phi) 
		+ g \, \phi \right) \ dx' \ dt = 0 \label{dyn_weakform1}
		\end{align}
		for every $\phi \in L^2(0,T'; H_0^2(S)) \cap H_0^1(0,T';L^2(S))$, and
		\begin{align}
			\int_0^{T'} \int_S \mathcal{B}^\nu_{{\rm vK},2}(u, v, \Psi)  \ dx' \ dt = 0 \label{dyn_weakform2} 
		\end{align}
		for every $\Psi \in L^2(0,T'; H_0^1(S; \R^2))$.
	\end{definition}

    \begin{remark}
    $Q_{\cell}^{\rel}$ induces a quadratic form on $\R^{2 \times 2}$ via $Q_2(A) = Q_{\cell}^{\rel}(\hat{A}Z)$ for $A \in \R^{2 \times 2}$. This is precisely the quadratic form of the classical von Kármán energy functional, cf.\ \cite{FrieseckeJamesMueller_06}. Its linearization $\mathcal{L}_2 = \frac{1}{2}Q_2$ satisfies $\mathcal{L}_2 A : B = \frac{1}{2}DQ_{\cell}^{\rel} (\hat{A} Z) \colon \hat{B} \, Z$ for $A, B \in \R^{2 \times 2}$. This shows that \eqref{dyn_weakform1} and \eqref{dyn_weakform2} are indeed the classical von Kármán equation if $\nu=\infty$, cp.\ \cite{MuellerPakzad_08,AbelsMoraMueller_09}. 
    \end{remark}

\subsection{Main result}

Now we have all the ingredients to state our main result.

\begin{theorem}\label{thm: main-td}
	Let $y_n^{(0)}, y_n^{(1)} : \tilde{\Lambra}_n \to \R^3$ be two sequences of lattice deformations satisfying 
	\begin{align}
		\frac{1}{2} \frac{\varepsilon_n^3}{h_n} \sum_{x \in \tilde{\Lambra}_n} \vert y_n^{(1)}(x)\vert^2 + \frac{\varepsilon_n^3}{h_n} \sum_{x \in \tilde{\Lambra}^\prime_n} W(x, \bar{\nabla}_n y_n^{(0)}(x))  \leq C h_n^4. \label{energy_ineq2}
	\end{align}
	as well as the boundary and compatibility conditions
	\begin{align*}
		y_n^{(0)}(x) = \matr{x' \\ h_n\left(x_3 - \frac{1}{2}\right)} \quad\text{and}\quad 
        y_n^{(1)}(x) = 0 
        \quad \text{for} \ x \in \tilde{\Lambra}_n \setminus \tilde{\Lambda}_{n,S}.
	\end{align*}
	Let $T \in (0,\infty]$ and $g \in L^2(I_T; W^{1,\infty}(S)) \cap C^0(I_T, L^\infty(S))$. Let $(y_n)$ be a sequence of solutions to the system of ordinary differential equations
	\begin{align}
		\begin{cases}
			h_n^2 \partial_t^2 y_n(t,x)  = - \bar{\nabla}_n^\ast \left( D_A W \left(\cdot, \bar{\nabla}_n y_n(t,\cdot)\right) \right)(x)  + h_n^3 g(t,x')e_3 \quad & \text{in} \ (0,T_n) \times \tilde{\Lambda}_{n,S}, \ \\
			y_n(t,x) = \left(x', h_n (x_3 - \frac{1}{2})\right) \quad &\text{on} \ (0,T_n) \times (\tilde{\Lambra}_n \setminus \tilde{\Lambda}_{n,S}) \\
			y_n(0,x) = y_n^{(0)}(x)  \quad &\text{in} \ \tilde{\Lambra}_n,\\
			\partial_t y_n(0,x) = h_n^{-1}y_n^{(1)}(x) \quad &\text{in} \ \tilde{\Lambra}_n,
		\end{cases} \tag{ODE} \label{ode}
	\end{align}
	where $T_n$ is the maximal time of existence. Then $T_n = T$ for all $n \in \N$.
	
	There exist $u \in L^\infty_{\mathrm{loc}}(I_T;H^1(S;\R^2))$ and
	$v \in L^\infty_{\mathrm{loc}}(I_T;H^2(S)) \cap W^{1,\infty}_\mathrm{loc}(0,T;L^2(S))$ such that, up to a subsequence, $u_n$ and $v_n$ as defined in \eqref{eq:un-def} and \eqref{eq:vn-def} satisfy
	\begin{align}
		u_n \overset{*}{\rightharpoonup} u \quad &\text{in} \ L^\infty_{\mathrm{loc}}(I_T; H^1(S;\R^2)), \label{conv-u-td} \\
		v_n \to v \quad &\text{in} \ L^\infty_{\mathrm{loc}}(I_T;L^2(S)), \label{conv-v-td} \\
		\partial_t v_n \overset{*}{\rightharpoonup} \partial_t v \quad &\text{in} \ L^\infty_{\mathrm{loc}}(I_T;L^2(S)) \label{conv-dv-td} . 
	\end{align}
	The maps $u$ and $v$ satisfy the boundary conditions
	\begin{equation}
		u_{|{\partial S} } = 0, \ v_{|{\partial S} } = 0, \ \nabla' v_{|{\partial S} } = 0.  \label{bc-td-limiting}
	\end{equation}
	Moreover the mappings
	\begin{align*}
		t \mapsto v(t): \quad & I_T \to H^2(S), \\
		t \mapsto \partial_t v(t): \quad & I_T \to L^2(S)
	\end{align*}
	are weakly continuous.
	The pair $(u,v)$ satisfies the equations \eqref{dyn_weakform1} and \eqref{dyn_weakform2} as well as the initial conditions
	\begin{align}
		v(0, x') = y_3^{(0)}(x'), \label{init_cond_limit_1} \\
		\partial_t v(0, x') = y_3^{(1)}(x')  \label{init_cond_limit_2} 
	\end{align}
	for almost every $x' \in S$,
	where
	\begin{align*}
		\frac{1}{h_n} \int_0^1 \left(\tilde{y}_n^{(0)} (\cdot, x_3)\right)_3 \ dx_3 
        &\to y_3^{(0)} \quad \text{in} \ H^1(S), \\
		\frac{1}{h_n^2} \int_0^1 \left(\tilde{y}_n^{(1)} (\cdot, x_3)\right)_3 \ dx_3 
        &\rightharpoonup y_3^{(1)} \quad \text{in} \ L^2(S).
	\end{align*}
\end{theorem}

    \begin{remark}\label{rmk:global-existence}
    Theorem~\ref{thm: main-td} entails global existence results for the classical von Kármán equation (which is well known, cp.\ \cite{KochStahel_93}) as well as the novel equation for ultrathin layers. Indeed, for given $y_3^{(1)} \in L^2(S)$ one can choose $\bar{y}_n^{(1)} : \tilde{\Lambra}_n \to \R^3$, not depending on the $x_3$-variable and vanishing on $\tilde{\Lambra}_n \setminus \tilde{\Lambda}_{n,S}$, such that its piecewise constant interpolation (see Section~\ref{subsec:Interpolation-schemes}) satisfies $\bar{y}_n^{(1)} \to y_3^{(1)}$ in $L^2$ and then set $y_n^{(1)} = h_n^2\bar{y}_n^{(1)}$. For given $y_3^{(0)} \in H^2_0(S)$ one may choose $y_n^{(0)}$ as a recovery sequence for the limiting out-of-plane variable $y_3^{(0)}$ and in-plane variable $0$ in the discrete-to-continuum $\Gamma$-convergence result proved in \cite{BraunSchmidt2022}, remarking that by density of $C^\infty_c(S)$ in $H^2_0(S)$ and construction one finds $y_n^{(0)}(x) = (x', h_n(x_3 - \frac{1}{2})$ for $x \in \tilde{\Lambra}_n \setminus \tilde{\Lambda}_{n,S}$. With these choices, \eqref{energy_ineq2} holds true and the solution $(u,v)$ from Theorem~\ref{thm: main-td} satisfies \eqref{dyn_weakform1} and \eqref{dyn_weakform2} with \eqref{bc-td-limiting}, \eqref{init_cond_limit_1} and \eqref{init_cond_limit_2}. 
    \end{remark}

\begin{example}
    Our result covers basic mass-spring models with nearest and next-to-nearest neighbor interactions given by
	\begin{align*}
		E_\text{atom}(w) 
		= \frac{1}{2} \sum_{\stackrel{x,x' \in \Lambda_{n}}{\abs{x-x'} = \varepsilon_n} } V_{\mathrm{NN}} \left( \frac{w(x) - w(x')}{\varepsilon_n}\right)
		+ \frac{1}{2} \sum_{\stackrel{x,x' \in \Lambda_{n}} {\abs{x-x'} = \sqrt{2} \varepsilon_n} } V_{\mathrm{NNN}} \left( \frac{w(x) - w(x')}{\varepsilon_n}\right),
	\end{align*}
	where $V_{\mathrm{NN}}, V_{\mathrm{NNN}} \in C^2(\R)$ are pair potentials such that for suitable $C, c > 0$
	\begin{enumerate}
		\item[(i)] $V_{\mathrm{NN}}(1) = 0 = V_{\mathrm{NNN}}(\sqrt{2})$,
		\item[(ii)] $V_{\mathrm{NN}}(r) \ge c (1-r)^2$ and $V_{\mathrm{NNN}}(r) \ge c (\sqrt{2}-r)^2$ all $r \in \R$,
		\item[(iii)] $\norm{V_{\mathrm{NN}}^\prime}_\infty + \norm{V_{\mathrm{NNN}}^\prime}_\infty \leq C$. 
	\end{enumerate}
    \end{example}

\section{Preparations} \label{sec:preparations}

	\subsection{Interpolation schemes}\label{subsec:Interpolation-schemes}
	To pass from discrete to continuum objects we need to interpolate in a suitable way. We shortly discuss the two interpolation schemes considered in \cite{BraunSchmidt2022}. 
    Let $w: \Lambra_n \to \R^3$ be a lattice mapping. For each $x \in \Lambra_n^\prime$ we 
    extend $w$ to $Q_n(x)$ in the following way. First set 
	\[\tilde{w}(x) = \frac{1}{8} \sum_{j=1}^{8} w(x+\varepsilon_n z^i). \] 
	Then, for the center points $v_1, \ldots, v_6$ of the six faces $F_1, \ldots, F_6$ of $[-\frac{1}{2}, \frac{1}{2}]^3$, we define
	\[ \tilde{w}(x+ \varepsilon_n v_k) \coloneqq \frac{1}{4} \sum_j w(x+\varepsilon_n z^j), \] 
	where the sum runs over corners $z_j$ of the face with center $v^k$. 
	Finally we interpolate linearly on each of the 24 simplices
	\[ \text{co}\big(x, x + \varepsilon_n v^k, x + \varepsilon_n z^i, x + \varepsilon_n z^j\big) \] 
	with $|z^i - z^j| =1, |z^i - v^k| = |z^j - v^k| = \frac{1}{\sqrt{2}}$. This defines a piecewise affine mapping $\tilde{w} \in H^1_{\mathrm{loc}}(\R^2 \times [0,1];\R^3)$. We note that if $w(x + \varepsilon_n z^i) = A z^i + c$ for all $i = 1,\ldots,8$ for some $A \in \R^{3 \times 3}$ and $c \in \R^3$, then also $\tilde{w}(\xi) = A \xi + c$ for $\xi \in Q_n(x)$. $\tilde{w}$ moreover satisfies
	\begin{align}
		\tilde{w}(x) &= \dashint_{Q_n(x)} \tilde{w}(\xi) \ d \xi \label{dashint-cube}\\
	\shortintertext{for $x \in \Lambra_n^\prime$ and}
		\tilde{w}(x + \varepsilon_n v^k) &= \dashint_{x + \varepsilon_nF^k} \tilde{w}(\zeta) \ d \zeta \label{dashint-face}
	\end{align}
	for every face $x + \varepsilon_n F^k$ of $Q_n(x)$, $x \in \Lambra_n^\prime$.
	We also observe another averaging property of this interpolation. 
	Let $Q_n'(x') = x' + (-\frac{\eps_n}{2},\frac{\eps_n}{2})^2$. Then for each $x \in \Lambra_n^\prime$ one has
	\begin{align}\label{eq:plane-average}
	 \dashint_{Q_n'(x')} \tilde{w}(\xi', \xi_3) \ d \xi' 
	 = \frac{1}{8} \sum_{j=1}^{8} w(x+\varepsilon_n z^j) 
	 + \frac{\xi_3 - x_3}{4\varepsilon_n} \bigg( \sum_{j=5}^{8} w(x+\varepsilon_n z^j) - \sum_{j=1}^{4} w(x+\varepsilon_n z^j) \bigg) 
	\end{align}
    for $\xi_3 \in (x_3 - \frac{\eps_n}{2}, x_3 + \frac{\eps_n}{2})$. This can be seen, e.g., by considering the rotated versions $w_0,\ldots, w_3$ of $w$ on $Q_n(x)$: If $R_{x}$ is the rotation about  $90$ degrees around $x + \R e_3$, we let $w_j(\xi) = w(R_{x}^j \xi)$, $j = 0,\ldots,3$, $\xi \in x + \{-\frac{\varepsilon_n}{2},\frac{\varepsilon_n}{2}\}^3$. Then their average $w_{\rm av} = \frac{1}{4} (w_0 + \ldots + w_3)$, which satisfies 
    \begin{align*}
     w_{\rm av} (\xi)
     = \begin{cases}
		\frac{1}{4} \sum_{j=1}^{4} w(x+\varepsilon_n z^j), & \xi \in x + \varepsilon_n \{z^1, \ldots, z^4\}, \\
			\frac{1}{4} \sum_{j=5}^{8} w(x+\varepsilon_n z^j), & \xi \in x + \varepsilon_n \{z^5, \ldots, z^8\},
		\end{cases}
    \end{align*}
    equals the right hand side of \eqref{eq:plane-average}. As this term is affine, we even have that $\tilde{w}_{\rm av} (\xi)$ is equal to the right hand side of \eqref{eq:plane-average} for each $\xi \in Q_n(x)$. Now \eqref{eq:plane-average} follows from the observation that our interpolation scheme yields 
    \begin{align*}
        \dashint_{Q_n'(x')} \tilde{w}(\xi',\xi_3) \ d\xi'
		= \frac{1}{4} \sum_{i=0}^3 \dashint_{Q_n'(x')} \tilde{w}_j(\xi',\xi_3) \ d\xi'
		= \dashint_{Q_n'(x')} \tilde{w}_{\rm av}(\xi',\xi_3) \ dx' 
		= \tilde{w}_{\rm av}(\xi). 
    \end{align*}
    From Lemma 3.2 in \cite{BraunSchmidt2022} we infer that there are constants $c, C > 0$ such that for every $x \in \Lambra_n^\prime$
	\begin{align}\label{eq: norm-eq-single-cube}
	 c |\bar{\nabla} w(x)|^2 \leq \dashint_{Q_n(x)} |\nabla \tilde{w}(\xi)|^2 \ d \xi \leq C |\bar{\nabla} w(x)|^2.
	\end{align}

	For the second interpolation we let 
	\[\bar{w} (\xi) = w(x) \quad \text{for every} \ \xi \in x + \left(- \frac{\varepsilon_n}{2}, \frac{\varepsilon_n}{2}\right)^3, \ x \in \Lambra_n .\] 
	Here we obtain a piecewise constant function $\bar{w} \in L^2_{\mathrm{loc}}(\R^2 \times [-\frac{\varepsilon_n}{2},1 + \frac{\varepsilon_n}{2}]; \R^3)$. Both interpolations have their own advantages. The first interpolation allows for an application of the results in \cite{FrieseckeJamesMueller_06} while for the second one the discrete gradient can be extended to an almost everywhere defined piecewise constant function on $\R^2 \times (0,h_n)$. This works as follows: For $\xi \in Q_n(x)$, $ x \in \Lambra_n^\prime$, we have
	\[ \xi + \varepsilon_n z^i \in \left(  x + \varepsilon_n z^i \right) + \left( - \frac{\varepsilon_n}{2}, \frac{\varepsilon_n}{2} \right)^3. \]
	We set
	\begin{align*}
		\bar \partial_i \bar w(\xi) := \bar w (\xi + \varepsilon_n z^i) - \frac{1}{8} \sum_{j=1}^{8} \bar w (\xi + \varepsilon_n z^j) \\
		= w (x + \varepsilon_n z^i) - \frac{1}{8} \sum_{j=1}^{8}  w (x + \varepsilon_n z^j)  
	\end{align*}
	which results in
	\[\bar{\nabla} \bar{w}(\xi) = \bar{\nabla} w(x) \quad \text{whenever} \ \xi \in Q_n(x), x \in \Lambra_n^\prime.\]
	    
	While the linear interpolation results in a quite regular function the piecewise constant interpolation is tailor-made to pass from sums to integral terms.
	The rescaled versions of the interpolated functions are defined as
	\begin{align*}
		\tilde{y} = \tilde{w}(H_n \,\cdot\,) \in H^1_{\mathrm{loc}}(\Oxc;\R^3) 
        \quad \text{and} \quad 
        \bar{y} = \bar{w}(H_n \ \cdot \ )\in L^2_{\mathrm{loc}}(\Oxc_{\nu_n};\R^3) , 
    \end{align*} 
    where $\Ox := \R^2 \times (0,1)$ and $\Ox_{\mu} := \R^2 \times (-\frac{1}{2(\mu - 1)}, 1 + \frac{1}{2(\mu -1)})$. 
	It is not hard to see that there is a constant $C > 0$ independent of $w$ and $n$ such that 
	\begin{align}\label{eq: interpol-bounds}
	 \|  {\tilde{y}} \|_{L^2(\Ox;\R^3)} 
	 \le C \| {\bar{y}} \|_{L^2({\Ox_{\nu_n}};\R^3)} 
	 \quad\text{and}\quad 
	 \| {\bar{y}} \|_{L^2(\Ox_{\nu_n};\R^3)} 
	 \le C \| {\tilde{y}} \|_{L^2(\Ox;\R^3)} 
	 \end{align}
	 (being finite whenever $w(x) = 0$ outside a bounded set). The correspondingly rescaled gradients will be denoted $\nabla_n \tilde{y} = (\nabla'\tilde{y}, h_n^{-1}\partial_3\tilde{y})$, where $\nabla'=(\partial_1,\partial_2$).
			
	Next we make precise what it means that a sequence of functions $f_n : \Lambra_n \to \R^m$ defined on the atomistic lattice converges to a corresponding limiting function. Natural choices for a continuum function space are 
	\begin{align*}
	 L^2_{\mathrm{loc}} \big( \Oxc;\R^m \big) 
	 \qquad \text{and} \qquad 
	 L^2_\mathrm{loc} \big( \R^2 \times \big\{0, \tfrac{1}{\nu-1}, \ldots, \tfrac{\nu-2}{\nu-1}, 1\big\};\R^m \big) 
    \end{align*}
    for thin films with $\nu_n \to \infty$, respectively, ultrathin films with $\nu_n \equiv \nu$. In the second case we extend a function $f$ that belongs to this limiting space in between the layers in two different ways: Its interpolation $f \in L^2_{\mathrm{loc}}(\Oxc;\R^m)$ which is continuous in $x_3$ and affine in $x_3$ on the intervals $(\frac{i-1}{\nu-1}, \frac{i}{\nu-1})$, $i = 1, \ldots, \nu-1$, is denoted by the same symbol whereas we write $f^* \in L^2_{\mathrm{loc}}(\Oxc_{\nu};\R^m)$ for its extension to a function which is piecewise constant in $x_3$ such that $f^*(x) = f(x', \frac{i}{\nu-1})$ whenever $x_3 \in (\frac{2i-1}{2(\nu-1)},\frac{2i+1}{2(\nu-1)})$, $i = 0, \ldots, \nu - 1$.

    The following result is observed in \cite{Schmidt:06,BraunSchmidt2022}, a detailed proof may be found in \cite{Buchberger:24}.
	\begin{proposition}\label{prop: equiv-convergences}
		Let $f_n : \tilde{\Lambra}_n \to \R^m$. 
		\begin{enumerate}
		 \item[a)] Assume $\nu_n \to \infty$. For $f \in L^2_{\mathrm{loc}}(\Oxc;\R^m)$ we have 
		 \begin{align*}
		 	\bar f_n \to f \ \text{in} \ {L^2_{\mathrm{loc}}(\Oxc; \R^m)} 
			~\iff~ 
			\tilde{f}_n \to f \ \text{in} \ {L^2_{\mathrm{loc}}(\Oxc; \R^m)}. 	 
	 	\end{align*}
        \end{enumerate}
        \begin{enumerate}
		 \item[b)] Assume $\nu_n \equiv \nu \in \N$. For $f \in L^2_{\mathrm{loc}} (\R^2 \times \{0, \tfrac{1}{\nu-1}, \ldots, \tfrac{\nu-2}{\nu-1}, 1\};\R^m) $ we have
		 \begin{align*}
		 	\bar f_n \to f^* \text{ in } L^2_{\mathrm{loc}} (\R^2 \times \{0, \tfrac{1}{\nu-1}, \ldots, \tfrac{\nu-2}{\nu-1}, 1\};\R^m) 
		 	~\iff~
		 	\tilde{f}_n \to f \ \text{in} \ {L^2_{\mathrm{loc}}(\Oxc; \R^m)}.
		 \end{align*}
        \end{enumerate}
	\end{proposition}

	The next lemma is a version of Proposition~\ref{prop: equiv-convergences} for weak-$\ast$-convergence in a time dependent setting.
	\begin{lemma}\label{lemma: weak-*-interpolations}
		Let $f_n: [0,T] \times \tilde{\Lambra}_n \to \R^m$.
		\begin{enumerate}
			\item[a)] Assume $\nu_n \to \infty$. For $f \in L^{\infty}(0,T; L^2_\mathrm{loc}(\Oxc;\R^m))$ we have 
			\begin{align*}
			{\bar{f}}_n \overset{\ast}{\rightharpoonup} f \text{ in } L^\infty(0,T; L^2_\mathrm{loc}(\Oxc;\R^m) ) 
			~\iff~ 
			{\tilde{f}}_n \overset{\ast}{\rightharpoonup} f \text{ in } L^\infty(0,T; L^2_\mathrm{loc}(\Oxc;\R^m)) .
            \end{align*}
			\item[b)] Assume $\nu_n \equiv \nu \in \N$. For $f \in L^{\infty}(0,T; L^2_\mathrm{loc} (\R^2 \times \{0, \tfrac{1}{\nu-1}, \ldots, \tfrac{\nu-2}{\nu-1}, 1\};\R^m))$ we have
			\begin{align*}
			&{\bar{f}}_n \overset{\ast}{\rightharpoonup} f^* \text{ in } L^{\infty}(0,T; L^2_\mathrm{loc} (\R^2 \times \{0, \tfrac{1}{\nu-1}, \ldots, \tfrac{\nu-2}{\nu-1}, 1\};\R^m)) \\ 
			&\quad \iff~ 
			{\tilde{f}}_n \overset{\ast}{\rightharpoonup} f \text{ in } L^\infty(0,T;L^2_\mathrm{loc}(\Oxc;\R^m))  .
            \end{align*}
        \end{enumerate}
	\end{lemma}

	\begin{proof}
		a) By \eqref{eq: interpol-bounds} the sequence $(\bar{f}_n)$ is bounded in $L^{\infty}(0,T; L^2(K;\R^m))$ for every compact set $K \subset \Oxc$ if and only if $(\tilde{f}_n)$ is bounded in $L^{\infty}(0,T; L^2(K \times [0,1];\R^m))$ for every compact $K \subset \Oxc$. 
		It is thus sufficient to consider test functions of the form $\varphi(t,x) = \eta(t) \chi_Q(x) \cdot e_i$ with $\eta \in C_c^\infty(0,T)$ and a cube $Q = \prod_{i=1}^3 [a_i, b_i) \subset \subset V \times (0,1)$ where $V$ is an arbitrary compact subset of $\R^2$.
		By our interpolation schemes we have, see \eqref{dashint-cube}, 
		\[ \dashint_{\tilde{Q}_n(x)} {\tilde{f}} \ d \xi 
		= {\tilde{f}}(x) 
		=  \frac{1}{8} \sum_{i=1}^{8} \dashint_{\tilde{Q}_n(x' + \varepsilon  (z^i)^\prime, x_3 + \frac{\varepsilon_n}{h_n} z^i_3)} {\bar{y}} \ d \xi \] for $x \in \tilde{\Lambra}^\prime_n$. So with $U_n = \prod_{i=1}^3 [a_i-\varepsilon_n, b_i+\varepsilon_n/h_n) \setminus \prod_{i=1}^3 [a_i+\varepsilon_n, b_i-\varepsilon_n/h_n)$ we have
		\begin{align*}
            \biggl \vert \int_Q ({\tilde{f}} - \bar f) \ d \xi \biggr \vert 
            \leq C \int_{U_n} |{\tilde{f}}| + |\bar f| \ d \xi 
            \le C (\| {\tilde{f}} \|_{L^2(V \times [0,1])} +\| {\tilde{f}} \|_{L^2(V \times [0,1])}) \sqrt{|U_n|}  
        \end{align*}	
		by Hölder's inequality. Hence, 
		\begin{align*}
			&\abs{\int_0^T \int_{V \times [0,1]} ({\tilde{f}}_n - {\bar{f}}_n) \varphi \ dx \ dt}
			= \abs{\int_0^T \eta(t) \int_Q ({\tilde{f}}_n - {\bar{f}}_n)\cdot e_i \ dx \ dt} 
			\le C \sqrt{|U_n|}  \to 0 
		\end{align*}
        as $n \to \infty$.
		
		b) Similarly as above we see that on each slice $x_3 = \frac{i}{\nu - 1}$, $i = 0, \ldots, \nu-1$, we have 
		\[{\bar{f}}_n(\cdot, x_3) \overset{\ast}{\rightharpoonup} f^*(\cdot, x_3) = f(\cdot, x_3) 
		\iff 
		{\tilde{f}}_n(\cdot, x_3) \overset{\ast}{\rightharpoonup} f(\cdot, x_3) \] 
		in $L^\infty(0,T;L^2_\mathrm{loc}(\R^2;\R^m))$. 
		The first condition is easily seen to be equivalent to ${\bar{f}}_n \overset{\ast}{\rightharpoonup} f^*$ in $L^{\infty}(0,T; L^2_\mathrm{loc} (\R^2 \times \{0, \tfrac{1}{\nu-1}, \ldots, \tfrac{\nu-2}{\nu-1}, 1\};\R^m))$. 
		
		To further investigate the second condition, we define $P_n^\prime \colon L^2(\R^2) \to L^2(\R^2)$ by
		\[ P_n^\prime g(\xi') = \dashint_{x' + \left( -\frac{\varepsilon_n}{2}, \frac{\varepsilon_n}{2}\right)^2} g(y') \ dy' \quad \text{whenever} \ \xi' \in x' + \left(-\frac{\varepsilon_n}{2}, \frac{\varepsilon_n}{2}\right)^2\]
        with $x' \in \varepsilon_n \Z^2$ and note that, by boundedness of $({\tilde{f}}_n)$ in $L^\infty(0,T;L^2(V \times [0,1];\R^m)$ for every compact subset $V \subset \R^2$, 
        \[ P_n^\prime {\tilde{f}}_n - {\tilde{f}}_n \overset{\ast}{\rightharpoonup} 0 
        \quad\text{and}\quad 
        P_n^\prime {\tilde{f}}_n\Big(\cdot, \frac{i}{\nu - 1}\Big) - {\tilde{f}}_n\Big(\cdot, \frac{i}{\nu - 1}\Big) \overset{\ast}{\rightharpoonup} 0. \]
        As for almost every $x' \in S$ the map $ x_3 \mapsto  P_n^\prime {\tilde{f}}_n(x',x_3)$ is affine on the intervals $(\frac{i-1}{\nu - 1}, \frac{i}{\nu - 1})$, $i = 1,\ldots,\nu -1$, by \eqref{eq:plane-average}, we see that ${\tilde{f}}_n(\cdot, \frac{i}{\nu - 1}) \overset{\ast}{\rightharpoonup} f(\cdot, \frac{i}{\nu - 1})$ in $L^\infty(0,T;L^2_\mathrm{loc}(\R^2;\R^m))$ holds true for all $i = 0,\ldots,\nu -1$ if and only if ${\tilde{f}}_n \overset{\ast}{\rightharpoonup} f$ in $L^{\infty}(0,T; L^2_\mathrm{loc} (\Oxc;\R^m))$. 
    \end{proof}

\subsection{Some properties of the quadratic forms and their derivatives}
	%
	
	We first note that by minimality on $SO(3)Z$, respectively, $SO(3)Z^{(1)}$, $DW_{\cell}(Z) = 0$ and $DW_{\surf}(Z^{(1)}) = 0$. 
	If $A \in \R_{\skewo}^{3 \times 3}$, so that $\exp(tA) \in SO(3)$ for all $t \in \R$, by frame indifference we  have $W_{\cell} (\exp(tA)Z) = W_{\surf} (\exp(tA)Z^{(1)}) = 0$ for all $t$. Differentiating at $t = 0$ leads to 
	\[Q_{\cell} \left(Z\right)\left[AZ,AZ\right] = 0 
		\qquad\text{as well as}\qquad  
		Q_{\surf} (Z^{(1)}) [AZ^{(1)}, AZ^{(1)}] = 0.\] 
	In particular, this implies 
	\[ Q_{\cell}^{\rel}(AZ)  = 0 
	 \quad\text{for } A \in \R_{\skewo}^{3 \times 3}
    \]
	As the quadratic forms $Q_{\cell}$, $Q_{\cell}^{\rel}$ and $Q_{\surf}$ are non-negative, this entails that 
	\begin{align}\label{eq: derivative-skew} 
	 DQ_{\cell}(AZ) 
	 = DQ_{\cell}^{\rel}(AZ) = 0
    \qquad\text{and}\qquad  
	 DQ_{\surf}(AZ^{(1)}) = 0 
	\end{align}
    for	$A \in \R^{3 \times 3}_{\skewo}$ and that  
    \begin{align*}
		Q_{\cell} (A + BZ) = Q_{\cell}(A) 
    \qquad\text{and}\qquad  
		Q_{\surf} (A + BZ^{(1)}) = Q_{\surf}(A)
	\end{align*}
	for every $A \in \R^{3 \times 8}$, respectively, $A \in \R^{3 \times 4}$, and every $B \in \R^{3 \times 3}_{\skewo}$. In particular, we have 
	\[ Q_{\cell}^{\rel} (A) 
	= \min_{b \in \R^3} Q_{\cell}(A + \sym(b \otimes e_3)Z) 
	= Q_{\cell}(A + \sym(b(A) \otimes e_3)Z). \] 
	In order to relate the derivatives of $Q_{\cell}^{\rel}$ and $Q_{\cell}$ we note that, by \eqref{eq:c-min-char}, 
	\begin{align}\label{eq: derivative-Qrel-Qcell}
	 DQ_{\cell}^{\rel} (A)
	 &= DQ_{\cell} (A + (b(A) \otimes e_3)Z) : (\,\cdot\, + b(\,\cdot\,) \otimes e_3)Z \nonumber \\ 
	 &= DQ_{\cell} (A + (b(A) \otimes e_3)Z). 
	\end{align}
    for every $A \in \R^{3 \times 8}$. 
	
	The following lemma shows that these derivatives coincide if a certain orthogonality condition holds true. 
	\begin{lemma} \label{lem: derivative-relaxed}
		Let $F \in \R^{3 \times 3}$ and $A \in \R^{3 \times 8}$. Suppose $F''$ denotes the upper left $2 \times 2$ submatrix of $F$. If
		\[ D Q_{\cell} \left( FZ + A\right) \perp \left(\R^3 \otimes e_3\right)Z \]
		then
		\[D Q_{\cell}^{\rel} (\hat{F}'' \, Z + A) 
		= DQ_{\cell} \left(FZ + A\right).  \]
	\end{lemma}
	
	\begin{proof}
		Let $F = \hat{F}'' + F'$ and let
		$c = (F_{13} + F_{31}, F_{23} + F_{32}, 2 F_{33})^T$. 
		Then we have
		$ \sym F'  = \sym (c \otimes e_3)$ and, since $DQ_{\cell}$ vanishes on $\R^{3 \times 3}_\text{skew}Z$,
		\begin{align*}
			DQ_{\cell} (FZ + A) 
			&= DQ_{\cell} ( \hat{F}'' Z + A + F'Z ) \\
			&= DQ_{\cell} ( \hat{F}'' Z + A + (c \otimes e_3)Z )
		\end{align*}
		is orthogonal to $(\R^3 \otimes e_3)Z$. But then \eqref{eq:c-min-char} implies $c = b(\hat{F}''  Z + A )$. 
		It follows form \eqref{eq: derivative-Qrel-Qcell} that 
		\[DQ_{\cell}^{\rel} ( \hat{F}'' Z + A  ) 
			= DQ_{\cell} ( \hat{F}'' Z + A + (c \otimes e_3)Z ) 
			= DQ_{\cell} (FZ + A ).\qedhere
		\]
	\end{proof}

	In order to check for the orthogonality relation in the assumptions of the previous results, the following observation is useful. 
	\begin{lemma}\label{lemma: surf-to-bulk-orth}
		Let $A \in \R^{3 \times 8}$ such that
		\begin{equation}
			\sum_{l = 1}^{4} \big[D Q_{\cell} (A)_{\cdot l} 
			+  DQ_{\surf} (A^{(1)})_{\cdot l}  \big] 
			= \sum_{l=5}^{8} DQ_{\cell} (A)_{\cdot l} \label{tag1}
		\end{equation}
		or
		\begin{equation}
			\sum_{l = 1}^{4} D Q_{\cell} (A)_{\cdot l} = 
			\sum_{l=5}^{8} DQ_{\cell} (A)_{\cdot l}
			+ \sum_{l=1}^{4} D Q_{\surf} (A^{(2)})_{\cdot l} \label{tag2}.
		\end{equation}
		Then 
		\[DQ_{\cell} (A) \perp \left(\R^3 \otimes e_3\right)Z. \]
	\end{lemma}
	
	\begin{proof}
		Assume that \eqref{tag1} holds true, the statement for \eqref{tag2} is shown similarly.	Let $Q: \R^{3 \times 8} \to \R$ be the quadratic form defined by
		\[Q(H) = Q_{\cell}(H) + Q_{\surf}(H^{(1)}). \]
		Since $Q_{\cell}$ is positive definite on $(\R^3 \otimes e_3)Z$ and $Q_{\surf} \geq 0$ this is also the case for $Q$. 
		By assumption we have
		\[ \big[DQ_{\cell}(A)^{(1)} + DQ_{\surf}(A^{(1)}), DQ_{\cell} (A)^{(2)} \big] \perp (\R^3 \otimes e_3)Z.    \]
		For every $c \in \R^3$ we have $(c \otimes e_3) Z^{(1)}  = - \frac{1}{2}(c,c,c,c) \in \R^{3 \times 4}$ and, since $t \mapsto Q_{\surf}(A^{(1)} + t (c \otimes e_3) Z^{(1)})$ is constant in $t$, we find $DQ_{\surf}(A^{(1)}) \perp (\R^3 \otimes e_3)Z^{(1)}$. This implies $DQ_{\cell}(A) \perp \left(c \otimes e_3\right)Z$.
	\end{proof}

	For later use we remark here that, more generally than above, frame indifference implies that if $B \in \R_{\skewo}^{3 \times 3}$, then indeed $W_{\cell} (\exp(tB)A) = W_{\cell}(A)$ and $W_{\surf} (\exp(tB)A^{(1)}) = W_{\surf} (A^{(1)})$ for any $A \in \R^{3 \times 8}$ and for all $t$, so that differentiating at $t = 0$ leads to 
	\begin{align}\label{eq:W-sym-prop}
	 DW_{\cell} (A)A^T:B 
	 = DW_{\surf}(A^{(1)}) (A^{(1)})^T : B = 0
	\end{align}
    for all $A \in \R^{3 \times 8}$, $B \in \R_{\skewo}^{3 \times 3}$.

\subsection{Auxiliary results in the continuous setup}

    The purpose of this section is to collect a number of results for continuous (non-atomistic) thin plates in the von Kármán energy regime that will be needed in our analysis later.

	The following two propositions provide a collection of the (for us) most relevant results of \cite{FrieseckeJamesMueller_06}, \cite{LecumberryMueller_09} and \cite{AbelsMoraMueller_09}. 
	
	\begin{proposition}\label{prop: td-displacements}
		Let $y_n \in L^2(0,T; H^1_{\mathrm{loc}}(\Oxc; \R^3))$ be a sequence with
		\begin{align*}
			\partial_t y_n &\in L^2(0,T; L^2(\Ox; \R^3) ), \\
			\partial_{tt} y_n &\in L^2(0,T; H^{-1}(\Ox;\R^3)),
		\end{align*}
		such that for every $T' \in I_T$
		\begin{align}
			\esssup_{t \in [0,T']} \int_{\Ox} \abs{\partial_t y_n(t,x)}^2 \ dx \leq C(T') h_n^2, \label{bound-derivative}\\
			\esssup_{t \in [0,T']} \int_{\Ox} \dist^2(\nabla_n y_n(t,x), SO(3) ) \ dx \leq C(T') h_n^4.
		\end{align}
		and
		\begin{equation}
			y_n(t,x) = \matr{x' \\ h_n(x_3 - \frac{1}{2})} \quad \text{whenever} \ x' \in \R^2 \setminus S_n, \label{eq:yn-Sn-RWe}
		\end{equation}
        where $S \subset S_n \subset \{x \in \R^2 : \dist(x,S) \le 2 \eps_n\}$.
		Then there is an approximating sequence $R_n \subset L^\infty_{\mathrm{loc}}(I_T; H^1_\mathrm{loc}(\R^2; \R^{3 \times 3}))$ such that $R_n(t,x') \in SO(3)$ for almost every $(t,x') \in (0,T) \times \R^2$ and
		\begin{align}
			\esssup_{t \in [0,T']} \norm{ \nabla_n y_n(t, \cdot) - R_n(t, \cdot)}_{L^2(\Ox)} \leq C(T') h_n^2, \label{y_n - r_n} \\
			\esssup_{t \in [0,T']} \norm{\nabla' R_n(t,\cdot)}_{L^2(\R^2)}  \leq C(T') h_n, \label{Dy_n - R_n} \\
			\esssup_{t \in [0,T']} \norm{R_n(t,\cdot ) - Id}_{H^1(\R^2)} \leq C(T') h_n \label{r_n - id}
		\end{align}
		for every $T' \in I_T$. Moreover, the averaged scaled in- and out-of-plane displacements
		\begin{align}
			u_n(t,x') &\coloneqq h_n^{-2} \int_0^1 (y_n(t,x))' - x' \ d x_3, \label{eq: un-def} \\
			v_n(t,x') &\coloneqq h_n^{-1} \int_0^1 (y_n(t,x))_3 \ dx_3 \label{eq: vn-def} 
		\end{align}
		satisfy, up to a subsequence, the following convergence properties:
		\begin{align}
			u_n \overset{\ast}{\rightharpoonup} u \quad &\text{in} \ L^\infty_{\mathrm{loc}}(I_T; H^1(\R^2;\R^2)), \label{weak-*-un}\\
			v_n \to v \quad  &\text{in} \ L^\infty_{\mathrm{loc}}(I_T; L^2(\R^2)), \label{strong-vn}\\
			\partial_t v_n \overset{\ast}{\rightharpoonup} \partial_t v \quad &\text{in} \ L^\infty_{\mathrm{loc}}(I_T; L^2(\R^2)). \label{weak-*-dtvn}
		\end{align}
		It even holds that $v \in L^\infty_{\mathrm{loc}}(I_T; H^2(\R^2))$. The maps $u$ and $v$ satisfy the boundary conditions
		\[ u = 0, \quad v = 0, \quad \nabla' v = 0 \qquad \text{on } \R^2 \setminus S.  \]
	\end{proposition}
    We remark that the construction of the maps $R_n$ in \cite{FrieseckeJamesMueller_06} yields $R_n = Id$ outside of a neighborhood of $S$, which allowed us to consider all quantities on the whole plane.

	For the proof of the following proposition we refer to the second step in the proof of Theorem 2.1, \cite{AbelsMoraMueller_09}.
	\begin{proposition}\label{prop: td-displacements-further}
		In the setting of Proposition~\ref{prop: td-displacements} with \eqref{weak-*-un}, \eqref{strong-vn} and \eqref{weak-*-dtvn} let 
		\[ A_n = h_n^{-1} \left( R_n - Id \right). \]
		Then
		\begin{equation}
			A_n \overset{\ast}{\rightharpoonup} A = e_3 \otimes \nabla' v - \nabla' v \otimes e_3 \quad \text{in} \ L^\infty_{\mathrm{loc}}(I_T; H^1(\R^2; \R^{3 \times 3})). \label{A_n-weak-*}
		\end{equation}
		The map $h_n^{-2} \sym(R_n - Id)$ is bounded in $L^\infty_{\mathrm{loc}}(I_T; L^p(\R^2; \R^{3 \times 3}))$ for every $p < \infty$. Moreover $A_n e_\alpha$ is strongly compact in $L^q_\mathrm{loc} (I_T; L^p(\R^2; \R^3))$ for $\alpha = 1,2$ and any $1 \leq p,q < \infty$. 
	\end{proposition}
    
    More generally than in \eqref{eq: un-def} and \eqref{eq: vn-def} we also identify the original, non-averaged scaled in- and out-of-plane displacements, cp.\ \cite{BraunSchmidt2022} in a stationary setting.

	\begin{proposition}\label{prop: td-displacements-3d}
		Let $y_n$ satisfy the assumptions of Proposition~\ref{prop: td-displacements}, assume \eqref{weak-*-un}, \eqref{strong-vn}, \eqref{weak-*-dtvn} hold true and define
		\begin{align*}
			\hat{u}_n(t,x) &\coloneqq h_n^{-2} \left( (y_n(t,x))' - x' \right), \\
			\hat{v}_n(t,x) &\coloneqq h_n^{-1} \left( y_n(t,x) \right)_3.
		\end{align*}
		Then
		\begin{align}
			\hat{u}_n \overset{\ast}{\rightharpoonup} \hat{u} \quad &\text{in} \ L^\infty_{\mathrm{loc}}(I_T; H^1(\Ox;\R^2)), \\
			\hat{v}_n \overset{\ast}{\rightharpoonup} \hat{v} \quad &\text{in} \ L^\infty_{\mathrm{loc}}(I_T; H^1_{\mathrm{loc}}(\Oxc)),
		\end{align}
		where
		\begin{align}
			\hat{u}(t,x) &= u(t,x') - \left(x_3 - \frac{1}{2}\right) \nabla'v(t,x'), \label{uhat} \\
			\hat{v}(t,x) &= v(t,x') + \left(x_3 - \frac{1}{2}\right). \label{vhat}
		\end{align}
		Moreover it even holds that
		\begin{equation}
			\hat{v}_n \to \hat{v} \quad \text{in} \ L^\infty_{\mathrm{loc}}(I_T; L^2_{\mathrm{loc}}(\Oxc)). \label{vn_3d_strong}
		\end{equation}
	\end{proposition}
	
	\begin{proof}
		Choosing $L > 0$ so large that $\bar S \subset \subset (-L,L)^2$ and applying Korn's inequality to $(\hat{u}_n^T(t), 0)^T \in H^1((-L, L)^2 \times (0,1);\R^3)$, which vanishes on $((-L, L)^2 \setminus S_n) \times (0,1)$, we get 
		\begin{align*}
			\norm{ \hat{u}_n(t) }_{H^1(\Ox; \R^2)} \leq
			C \big( \norm{ \sym \ \nabla' \hat{u}_n(t) }_{L^2(\Ox; \R^{2 \times 2})} + \norm{\partial_3 \hat{u}_n(t)}_{L^2(\Ox; \R^2)} \big)
		\end{align*}
		for almost every $t$. By \eqref{y_n - r_n} and Proposition~\ref{prop: td-displacements-further} we recognize that the first term on the right hand side is bounded uniformly in $t$. Also the second term is bounded uniformly in $t$ by \eqref{y_n - r_n} and \eqref{r_n - id} since $\partial_3 \left(\hat{u}_n\right)_i = h_n^{-1} \left(\nabla_n y_n - Id\right)_{i3}$ for $i = 1,2$. This proves boundedness of $\hat{u}_n$ in $L^\infty_{\mathrm{loc}}(I_T; H^1(\Ox; \R^2))$. It remains to identify the limit. Since
		\[ \int_0^1 \hat{u}_n \ dx_3 \overset{\ast}{\rightharpoonup} u \quad \text{in} \ L^\infty_{\mathrm{loc}}(I_T; H^1(\R^2; \R^2)) \]
		and, for $i = 1,2$,
		\[ \partial_3 \left(\hat{u}_n\right)_i = h_n^{-1} \left(\nabla_n y_n - Id\right)_{i3} \overset{\ast}{\rightharpoonup} - \partial_i v \quad \text{in} \ L^\infty_{\mathrm{loc}}(I_T;  L^2(\Ox)) \]
		by \eqref{y_n - r_n} and Proposition~\ref{prop: td-displacements-further}. Therefore the limit is identified by \eqref{uhat} and subsequences were not needed. 
		
		For the convergence of $\hat{v}_n$ it is convenient to consider $\hat{w}_n(t,x) := \hat{v}_n(t,x) - (x_3 - \frac{1}{2})$. From \eqref{y_n - r_n} and \eqref{r_n - id}
		we get the bound
		\begin{align}
			\esssup_{t \in [0,T']} \norm{\nabla \hat{w}_n(t)}_{L^2(\Ox)}  \leq C(T'). \label{tag}
		\end{align}
        Applying Poincaré's inequality to $\hat{w}_n \in H^1((-L, L)^2 \times (0,1);\R^3)$, which, for $L$ as before, vanishes on $((-L, L)^2 \setminus S_n) \times (0,1)$, we obtain boundedness of $\hat{w}_n$ in $L^\infty_{\mathrm{loc}}(0,T;L^2(\Ox))$ and, hence, $\hat{w}_n \overset{\ast}{\rightharpoonup} \hat{w}$ in $L^\infty_{\mathrm{loc}}(I_T; H^1(\Ox))$ for some $\hat{w} \in L^\infty_{\mathrm{loc}}(I_T; H^1(\Ox))$. To identify $\hat{w}$ note that by \eqref{strong-vn}
		\[ \int_0^1 \hat w_n \ dx_3 = \int_0^1 \hat v_n \ dx_3 \to v \]
		and by \eqref{y_n - r_n} and \eqref{r_n - id}
		\[ \partial_3 \hat w_n \to 0 \]
		both in $L^\infty_{\mathrm{loc}}(I_T; L^2(\Ox))$.
		Hence the limit $\hat{v}$ is given by \eqref{vhat}.
		
		Finally the convergence \eqref{vn_3d_strong} follows from the Aubin-Lions lemma, since, due to \eqref{bound-derivative}, the mappings $\partial_t \hat{v}_n$ are bounded in $L^\infty_{\mathrm{loc}}(I_T; L^2(\Ox))$.
	\end{proof}

	We finally state here a result which will allow us to linearize the stored energy around a weakly convergent sequence. The proof in (\cite{MuellerPakzad_08}, Proposition 2.3) in the static case can be adapted in a straightforward way. 
	\begin{proposition}\label{prop: ws-linearization}
		Let $E \subset \R^n$ be a bounded, measurable set, $1 < p < \infty$. Let $f: \R^n \to \R^n$ be a function which is differentiable at $0$ and satisfies for every $a \in \R^n$ the inequality
		$$\abs{f(a)} \leq C \abs a.$$
		Let $z^\delta \overset{*}{\rightharpoonup} z$ in $L^\infty\left(0,T;L^p(E;\R^n)\right)$. Then
		$$\frac{1}{\delta} f(\delta z^{(\delta)}) \overset{*}{\rightharpoonup} Df(0)z \text{ in } L^\infty\left(0,T;L^p(E;\R^n)\right).$$
	\end{proposition}

    \subsection{The discrete, time-dependent strain}
    
    Consider a sequence $y_n$ with
		\[\esssup_{t \in [0,T']} E_n (y_n(t)) \leq C(T') h_n^4 \quad \text{for all} \ T' \in (0,T]. \]  
   	We need to find convergence of the discrete strain
	\[\bar{G}_n = h_n^{-2} \left( R_n^T \bar{\nabla}_n \bar y_n - Z \right). \] 

    From \cite[Proposition~1]{BraunSchmidt2022} we infer that then
    \begin{align}
        0 \leq \esssup_{t \in [0,T']} \int_\Ox \dist^2\left( \nabla_n {\tilde{y}}_n(t), SO(3)\right) \ dx \leq C h_n^4, \label{eq:energy-bound-interpolations}
    \end{align}
	where ${\tilde{y}}_n \in H^1_{ \mathrm{loc}}(\Ox; \R^3)$ is the piecewise affine interpolation of $y_n$.
	Thus we can use the results of \cite{AbelsMoraMueller_09} proven in Step 4 of their Theorem 2.1: The sequence
	\[G_n \coloneqq h_n^{-2} \left(R_n^T \nabla_n {\tilde{y}}_n - Id \right) \]  
	converges weakly-* to some $G$ in $L^\infty_{\mathrm{loc}}(I_T; L^2({\Ox}, \R^{3 \times 3}))$.
	The upper $2 \times 2$ matrix $G''$ of $G$ is affine in $x_3$, i.e.
    \begin{equation}
        G''(t,x) = G_1(t, x') + \Big(x_3 - \frac{1}{2}\Big)G_2(t,x')  \label{eq:G1G2}
    \end{equation}
	with $\sym \ G_1 = \frac{1}{2} \left( \nabla' u + \nabla' u ^T + \nabla' v \otimes \nabla' v \right)$ and $G_2 = - \nabla'^2 v$.
	The mappings $u$ and $v$ are the ones from Proposition~\ref{prop: td-displacements}. 
	The main task in this section is to identify the full limiting strain in $\R^{3 \times 8}$ . 
	
	We use the projections $P_n$ to piecewise constant functions, defined by
	\[ P_nf(x) = \dashint_{\tilde{Q}_n(x)} f(\xi) \ d \xi, \quad x \in \tilde{\Lambra}^\prime_n, \]
	on $\tilde{Q}_n(x)$, and $P$ defined by
	\[ P f(x) = \dashint_{\frac{k-1}{\nu - 1}}^{\frac{k}{\nu - 1}} f(x',t) \ dt \quad \text{if} \ \frac{k-1}{\nu - 1} \leq x_3 \leq \frac{k}{\nu -1}, \ k \in \{1, \ldots, \nu - 1\} \]
	if $\nu_n \equiv \nu \in \N$ and $P f = f$ if $\nu_n \to \infty$. We recall that for every $f \in L^2(\Ox)$ we have $P_nf \to Pf$ in $L^2(\Ox)$. 	
	
	\begin{proposition}\label{prop: strain-td}
		Let $(y_n)_n$ be a sequence of discrete deformations such that for every $T' \in (0,T)$
		\begin{equation}
			\esssup_{t \in [0,T']} E_n (y_n(t)) \leq C(T') h_n^4 \label{energy-bound-td}
		\end{equation}
		and the interpolations ${\tilde{y}}_n$ satisfy the assumptions of Proposition~\ref{prop: td-displacements}.
		Let
		\[ h_n^{-2} \left(R_n^T \nabla_n {\tilde{y}}_n - Id \right) \overset{*}{\rightharpoonup} G \quad \text{in} \ L^\infty_{\mathrm{loc}}(I_T; L^2(\Ox, \R^{3 \times 3})).\]
		Then
		\[\bar{G}_n \overset{*}{\rightharpoonup} \bar{G} \quad \text{in} \ L^\infty_{\mathrm{loc}}(I_T; L^2(\Ox,\R^{3 \times 8})),\]
		where
		\[ \bar{G} =
		\begin{cases}
			GZ & \text{if} \ \nu_n \to \infty, \\
			PGZ + \frac{1}{2(\nu - 1)}(\hat{G}_2 Z_- + \partial_{12} v \, M ) & \text{if} \ \nu_n \equiv \nu \in \N.
		\end{cases}
		\]
	\end{proposition}
	The proof can be adapted from the corresponding result in the stationary setting discussed in 
	\cite{BraunSchmidt2022}. In the following we will omit all details of computations for a fixed time $t$ which can directly be inferred from \cite{BraunSchmidt2022}. 
	
	We start with the boundedness of $\bar{G}_n$ in $L^\infty_{\mathrm{loc}}(I_T;L^2(\Ox;\R^{3 \times 8}))$.
	
	\begin{lemma}\label{lem: strain-td-bound}
		The sequence $\bar{G}_n$ is bounded in $L^\infty_{\mathrm{loc}}(I_T;L^2(\Ox;\R^{3 \times 8}))$.
	\end{lemma}
	
	\begin{proof}
    By \eqref{eq: norm-eq-single-cube} we have for $x \in \tilde{\Lambra}_n^\prime$
		\begin{align*}
			\int_{\tilde{Q}_n(x)} |\bar{\nabla}_n \bar y_n - P_n R_n Z|^2 \ d \xi 
			&= \frac{\varepsilon_n^3}{h_n} |\bar{\nabla}_n \bar y_n(x) - P_n R_n(t,x') Z|^2 \\
			&\leq C \int_{\tilde{Q}_n(x)} |\nabla_n {\tilde{y}}_n - P_n R_n^2 |^2 \ d \xi.
		\end{align*}
    The Poincaré inequality on $\tilde{Q}_n(x)$ gives 
		\begin{align*}
			\int_{\tilde{Q}_n(x)} |(R_n - P_n R_n) Z|^2 \ d \xi 
  			&\leq C \int_{\tilde{Q}_n(x)} |R_n - P_n R_n|^2 \ d \xi 
  			\leq C h_n^2 \int_{\tilde{Q}_n(x)} |\nabla' R_n|^2 \ d \xi.
		\end{align*}
    Combining these estimates we arrive at 
		\begin{align*}
			\int_{\tilde{Q}_n(x)} |\bar{G}_n|^2 \ d \xi
			&= h_n^{-4} \int_{\tilde{Q}_n(x)} |\bar{\nabla}_n \bar y_n - R_n Z|^2 \ d \xi \\
			&\leq C h_n^{-4} \int_{\tilde{Q}_n(x)} |\nabla_n {\tilde{y}}_n(t,\xi) - R_n|^2 \ d \xi + C h_n^{-2} \int_{\tilde{Q}_n(x)} |\nabla' R_n|^2 \ d \xi. 
		\end{align*} 
    Now, since $\bar{G}_n = 0$ outside a neighborhood of $S$, summing over all cubes we obtain the claim from \eqref{y_n - r_n} and \eqref{Dy_n - R_n}.
   	\end{proof}

	\begin{proof}[Proof of Proposition~\ref{prop: strain-td}]
		Compactness follows from Lemma~\ref{lem: strain-td-bound}, i.e. there is a $\bar{G} \in L^\infty_{\mathrm{loc}}(I_T; L^2(\Ox; \R^{3 \times 8}))$ such that, up to a subsequence, $\bar{G}_n \overset{\ast}{\rightharpoonup} \bar{G}$ in $L^\infty_{\mathrm{loc}}(I_T; L^2(\Ox; \R^{3 \times 8}))$. It remains to identify this limit. As $R_n \to Id$ boundedly in measure on $(0,T') \times \Ox$ for every $T' \in (0,T)$ we have
		\[R_n \bar{G}_n = h_n^{-2}(\bar{\nabla}_n \bar y_n - R_n Z) \overset{\ast}{\rightharpoonup} \bar{G}\]
		in $L^\infty_{\mathrm{loc}}(I_T; L^2(\Ox; \R^{3 \times 8}))$ as well as
		\[h_n^{-2} (\nabla_n {\tilde{y}}_n - R_n) \overset{\ast}{\rightharpoonup} G \]
		in $L^\infty_{\mathrm{loc}}(I_T;L^2(\Ox;\R^{3 \times 3}))$. We distinguish between the affine parts and the non-affine parts. A $b \in \R^8$ is called affine if it is in the linear span of the vectors $b^0, \ldots, b^3$ with $b^0 = (1, \ldots, 1)$ and $b^i = Z^T e_i$ for $i \in \{1,2,3\}$. All vectors orthogonal to the affine vectors are called non-affine. They are characterized by $\sum_{i=1}^{8} b_i  = 0$ and $Zb = 0$.
		We start with the affine parts. For every $n \in \N$ we have $R_n \bar{G}_n b^0 = 0$ and therefore $\bar{G}b^0 = 0$. For $t$ fixed we get from the calculations in \cite{BraunSchmidt2022}, Proposition 4 that $P_n[R_n  \bar{G}_n] b^i
		= \frac{2}{h_n^2} P_n [ \partial_i {\tilde{y}}_n - R_n e_i ]$  for $i = 1,2$ and $P_n[R_n  \bar{G}_n] b^3
		= \frac{2}{h_n^2} P_n [h_n^{-1} \partial_3 {\tilde{y}}_n - R_n e_3 ]$ for $i = 3$. Thus
		\begin{align*}
			P_n[R_n  \bar{G}_n] b^i
			\overset{\ast}{\rightharpoonup} 2 PG e_i = PGZ b^i
		\end{align*}
		for $i = 1,2,3$. Summarized, for every affine vector $b$ we have $\bar{G} b = PGZb$.
		
		For a non-affine vector $b$ we write $b^{(1)}=(b_1,b_2,b_3,b_4)^T$ and $b^{(2)}=(b_5,b_6,b_7,b_8)^T$. Further we consider the two-dimensional difference operator
		\[ \bar{\nabla}_n^{2d} f(x) \coloneqq \frac{1}{\varepsilon_n} \bigg( f(x' + \varepsilon_n(z^i)', x_3) - \frac{1}{4} \sum_{j=1}^{4} f(x' + \varepsilon_n (z^j)', x_3) \bigg)_{i = 1,2,3,4}. \]
		By $(\bar{\nabla}_n^{2d})^\ast$ we denote its formal adjoint determined by the relation
		\[ \sum_{ x \in \tilde{\Lambra}^\prime_n} \bar{\nabla}_n^{2d} f(x) \colon H(x) = \sum_{ \omega \in \tilde{\Lambra}_n} f(\omega) \cdot \big((\bar{\nabla}_n^{2d})^\ast H \big) (\omega) \]
		for every $H: \tilde{\Lambra}^\prime_n \to \R^{3 \times 4}$ with compact support.
		Then, as shown in \cite{BraunSchmidt2022}, (44) and (45),
		\begin{align}
			R_n \bar{G}_n(t,x) b 
			&= h_n^{-2} \bar{\nabla}_n \bar{y}_n(t,x) b \nonumber \\
			&= h_n^{-2} \Big( \bar{\nabla}_n^{2d} \bar{y}_n\Big(t,x+ \frac{\varepsilon_n}{2h_n} e_3\Big) - \bar{\nabla}_n^{2d} \bar{y}_n\Big(t,x - \frac{\varepsilon_n}{2h_n} e_3\Big)  \Big) b^{(2)}  \label{td-strain-1}\\
			&\qquad + h_n^{-2} \bar{\nabla}_n^{2d} \bar{y}_n\Big(t,x- \frac{\varepsilon_n}{2h_n} e_3\Big) \big( b^{(1)} + b^{(2)} \big) \label{td-strain-2} .
		\end{align}
		We treat \eqref{td-strain-1} and \eqref{td-strain-2} separately and start with \eqref{td-strain-2}. Let $\varphi(t,x) = \eta(t) \Psi(x)$ with $\eta \in C_c^\infty(0,T')$ and $\Psi \in C_c^\infty(\Ox; \R^3)$. Then, with $Z^{2d} = \bigl((z^1)^\prime, (z^2)^\prime, (z^3)^\prime, (z^4)^\prime \bigr) \in \R^{2 \times 4} $, 
		\[ \big(\bar{\nabla}_n^{2d}\big)^\ast \varphi(t,x) \to - \nabla' \varphi(t,x)Z^{2d} \]
		uniformly and therefore for $i \in \{1,2\}$, by Proposition~\ref{prop: td-displacements-3d}, Lemma~\ref{lemma: weak-*-interpolations} and due to $Z^{2d}(b^{(1)} + b^{(2)}) = 0$,
		\begin{align}
            \MoveEqLeft
			h_n^{-2} e_i^T \int_0^{T'} \int_\Ox \bar{\nabla}_n^{2d} \big(\bar{y}_n - \bar{\rm id} \big)\Big(t,x- \frac{\varepsilon_n}{2h_n} e_3\Big) \big( b^{(1)} + b^{(2)} \big) \varphi(t,x) \ dx \ dt \nonumber \\
			&= h_n^{-2} e_i^T \int_0^{T'} \int_\Ox \big(\bar{y}_n - \bar{\rm id} \big)\Big(t,x- \frac{\varepsilon_n}{2h_n} e_3\Big) \big(\bar{\nabla}_n^{2d}\big)^\ast \varphi(t,x) \big(b^{(1)} + b^{(2)}\big) \ dx \ dt  \nonumber \\
			&\to - \int_0^{T'} \int_\Ox \hat{u}_i(t,\tilde{x}) \nabla' \varphi(t,x) Z^{2d} \big( b^{(1)} + b^{(2)} \big) \ dx \ dt = 0,\label{count1}
		\end{align}
		where $\tilde{x} = x$ if $\nu_n \to \infty$ and $\tilde{x} = (x', \frac{\lfloor (\nu - 1) x_3 \rfloor }{\nu - 1})$. For $i = 3$ we have
		\begin{align}
			\MoveEqLeft 
			h_n^{-2} e_3^T \int_0^{T'} \int_\Ox \bar{\nabla}_n^{2d} \bar{y}_n\Big(t,x-\frac{\varepsilon_n}{2h_n}e_3\Big) \big( b^{(1)} + b^{(2)} \big) \varphi(t,x) \ dx \ dt \nonumber \\
			&= \frac{\varepsilon_n}{h_n} \int_0^{T'} \int_\Ox \frac{\left(\bar{y}_n\right)_3 (t, x - \frac{\varepsilon_n}{2h_n} e_3 )}{h_n} \cdot \frac{\big( \bar{\nabla}_n^{2d} \big)^\ast \varphi(t,x)  
			+ \nabla_n^\prime \varphi(t,x) Z^{2d}}{\varepsilon_n} \big( b^{(1)} + b^{(2)} \big) \ dx \ dt. \label{bla}
		\end{align}
		Now
		\begin{align}
            \MoveEqLeft
			\frac{1}{\varepsilon_n} \big( \big(\bar{\nabla}_n^{2d} \big)^\ast \varphi(t,x) + \nabla_n^\prime \varphi(t,x) Z^{2d} \big) \nonumber \\
			&\to \bigg( \frac{1}{2} {\nabla'}^2 \varphi(t,x) \left[ (z^i)', (z^i)' \right]
			- \frac{1}{8} \sum_{j=1}^{4} {\nabla'}^2 \varphi(t,x) \left[(z^j)', (z^j)' \right]
			\bigg)_{i = 1,2,3,4} \label{verweis}
		\end{align}
		uniformly. By Proposition~\ref{prop: td-displacements-3d} together with Proposition~\ref{prop: equiv-convergences} if $\nu_n \to \infty$ from \eqref{bla} we get
		\begin{equation}
			h_n^{-2} e_3^T \int_0^{T'} \int_\Ox \bar{\nabla}_n^{2d} \bar{y}_n\Big(t,x - \frac{\varepsilon_n}{2h_n} e_3 \Big) \big(b^{(1)} + b^{(2)} \big) \varphi(t,x) \ dx \ dt \to 0. \label{count2}
		\end{equation}
		For $\nu_n \equiv \nu \in \N$ instead by Proposition~\ref{prop: td-displacements-3d} and Proposition~\ref{prop: equiv-convergences} it follows from \eqref{bla} and \eqref{verweis} that
		\begin{align}
            \MoveEqLeft
			h_n^{-2} e_3^T \int_0^{T'} \int_\Ox \bar{\nabla}_n^{2d} \bar{y}_n\Big(t,x - \frac{\varepsilon_n}{2h_n} e_3 \Big) \big(b^{(1)} + b^{(2)} \big) \varphi(t,x) \ dx \ dt \nonumber \\
			&\to \frac{1}{\nu - 1} \int_0^{T'} \int_\Ox \Big( \frac{1}{2}  {\nabla'}^2 v \big[ (z^i)', (z^i)' \big] \Big)_{i = 1,2,3,4} \big( b^{(1)} + b^{(2)} \big) \varphi(t,x) \ dx \ dt. \label{count3}
		\end{align}
		For \eqref{td-strain-1} let $\varphi$ as above. As shown in \cite{BraunSchmidt2022}, for fixed $t$ and integrating afterwards it holds that 
		\begin{align*}
            \MoveEqLeft
			\int_0^{T'} \int_\Ox h_n^{-2} \Big( \bar{\nabla}_n^{2d} \bar{y}_n\Big(t,x+ \frac{\varepsilon_n}{2h_n} e_3\Big) - \bar{\nabla}_n^{2d} \bar{y}_n\Big(t,x - \frac{\varepsilon_n}{2h_n} e_3\Big)  \Big) b^{(2)} \cdot \varphi(t,x) \ dx \ dt \\
			&= \frac{\varepsilon_n}{h_n} \int_0^{T'} \int_{\Ox} P_n \tilde{A}_n(t,x) e_3 \cdot \big(\bar{\nabla}_n^{2d}\big)^\ast \varphi(t,x) \, b^{(2)} \ dx \ dt
		\end{align*}
		with $\tilde{A}_n = \frac{\nabla_n \tilde{y}_n - Id}{h_n}$. Since $\tilde{A}_n \overset{\ast}{\rightharpoonup} A = e_3 \otimes \nabla'v - \nabla'v \otimes e_3$ in $L^\infty_{\mathrm{loc}}(I_T; L^2(\Ox; \R^{3 \times 3}))$ by \eqref{y_n - r_n} and \eqref{A_n-weak-*} we have $P_n \tilde{A}_n \overset{\ast}{\rightharpoonup} PA$ and thus
		\begin{equation}
			\int_0^{T'} \int_\Ox h_n^{-2} \Big( \bar{\nabla}_n^{2d} \bar{y}_n\Big(t,x+ \frac{\varepsilon_n}{2h_n} e_3\Big) - \bar{\nabla}_n^{2d} \bar{y}_n\Big(t,x - \frac{\varepsilon_n}{2h_n} e_3\Big)  \Big) b^{(2)} \cdot \varphi(t,x) \ dx \ dt \to 0 \label{count5}
		\end{equation}
		if $ \nu_n \to \infty$ and
		\begin{align}
            \MoveEqLeft
			\int_0^{T'} \int_\Ox h_n^{-2} \Big( \bar{\nabla}_n^{2d} \bar{y}_n\Big(t,x+ \frac{\varepsilon_n}{2h_n} e_3\Big) - \bar{\nabla}_n^{2d} \bar{y}_n\Big(t,x - \frac{\varepsilon_n}{2h_n} e_3\Big)  \Big) b^{(2)} \cdot \varphi(t,x) \ dx \ dt \nonumber \\
			&\to - \frac{1}{\nu - 1} \int_0^{T'} \int_\Ox PA e_3 \cdot \nabla' \varphi(t,x) Z^{2d} b^{(2)} \ dx \ dt \nonumber \\
			&= - \frac{1}{\nu - 1} \int_0^{T'} \int_\Ox \matr{ \nabla'^2 v(t,x') Z^{2d} b^{(2)} \\ 0 } \cdot \varphi(t,x) \ dt \ dx.   \label{count4}
		\end{align}																																
		This finishes the investigations of the relevant convergences. Summarized, for every non-affine $b \in \R^8$ in case $\nu_n \to \infty$ we get $\bar{G}b = 0$ by \eqref{count1}, \eqref{count2} and \eqref{count5}. If $\nu_n \equiv \nu \in \N$ by \eqref{count1}, \eqref{count3} and \eqref{count4} we obtain $\bar{G}b = (PGZ + \frac{1}{2(\nu - 1)} (\hat{G}_2 Z_- + \partial_{12} v \, M ))b$ for every non-affine $b \in \R^8$ after repeating the calculations of \cite{BraunSchmidt2022} for fixed $t$ and integrating in time afterwards. Thus for every $b \in \R^8$ it holds that $\bar{G}b = GZ$ if $\nu_n \to \infty$ and $\bar{G}b = (PGZ + \frac{1}{2(\nu - 1)}(\hat{G}_2 Z_- + \partial_{12} v \, M ))b$ for every $b \in \R^8$.
	\end{proof}

\subsection{Consequences of the equations of motion}

	Throughout this section we assume that $y_n$ is a sequence of (extended) discrete deformations satisfying \eqref{dyn_weak_sltn} as well as the energy bounds 
	\begin{align}
		\esssup_{t \in [0,T']} E_n(y_n(t)) \leq C(T') h_n^4, \\
		\esssup_{t \in [0,T']} \frac{\varepsilon_n^3}{h_n} \sum_{x \in \tilde{\Lambra}_n } \abs{\partial_t y_n(t,x)}^2 \leq C(T') h_n^2\label{bound-derivative-td}
	\end{align}
	for every $T' \in I_T$. 
	We define
	\begin{align*}
		J^n(t,x) &\coloneqq h_n^{-2} DW_{\cell} (Z + h_n^2 \bar{G}_n(t,x)), \\
		J^{(j,n)}(t,x) &\coloneqq h_n^{-2} DW_{\surf}(Z^{(1)} + h_n^2 \bar{G}_n^{(j)}(t,x)), \quad j = 1,2. 
	\end{align*}
	It is an immediate consequence of the growth conditions on $DW_{\cell}$ and $DW_{\surf}$ that all of these mappings are bounded in $L^\infty_{\mathrm{loc}}( I_T; L^2(\Ox;\R^{3 \times 8}))$. By Proposition~\ref{prop: ws-linearization} we have the convergences
	\begin{align}
		J^n \overset{\ast}{\rightharpoonup} J \coloneqq D^2W_{\cell}(Z)[\bar{G}] \quad &\text{in} \ L^\infty_{\mathrm{loc}}(I_T; L^2(\Ox; \R^{3 \times 8})), \label{conv-Jn} \\
		J^{(j,n)} \overset{\ast}{\rightharpoonup} J^{(j)} \coloneqq  D^2W_{\surf}(Z^{(1)})[\bar{G}^{(j)}] \quad &\text{in} \ L^\infty_{\mathrm{loc}}(I_T; L^2(\Ox; \R^{3 \times 4})), \quad j = 1,2. \label{conv-Jnsurf}
	\end{align}
	We combine these contributions as 
    \begin{align*}
        \tilde{J}^n(t,x) 
		&:= h_n^{-2} D_A W(x_3, Z + h_n^2 \bar{G}_n(t,x)) \\
		&= J^n(t,x) + \chi_{(0, \frac{1}{\nu_n - 1})}(x_3) (J^{(1,n)}(t,x), 0) + \chi_{(\frac{\nu_n - 2}{\nu_n - 1},1 )}(x_3) (0, J^{(2,n)}(t,x))
    \end{align*}
    and write
    \begin{align}\label{eq: cor: scaled-strain-td-1}
        \tilde J(t,x) =
        J(t,x) + \big( (\chi_{( 0, \frac{1}{\nu - 1})}(x_3) J^{(1)}(t,x), \chi_{( \frac{\nu - 2}{\nu - 1},1)}(x_3) J^{(2)}(t,x) \big)
    \end{align}
    for $\nu \in \{2, 3, \ldots, \infty\}$, which equals $J(t,x)$ if $\nu = \infty$ and is piecewise constant in $x_3$ in case $\nu \in \N$. By \eqref{conv-Jn} and \eqref{conv-Jnsurf} we have  
	\begin{align}\label{eq: cor: scaled-strain-td-2}h_n
        ^{-2} D_A W(\cdot, Z + h_n^2 \bar{G}_n(\cdot)) \overset{\ast}{\rightharpoonup} \tilde{J} \quad \text{in} \ L^\infty_{\mathrm{loc}}(I_T; L^2(\Ox; \R^{3 \times 8})).
	\end{align}

	It is useful to write the weak form of the equations of motion \eqref{dyn_weak_sltn} in terms of $\tilde{J}^n$. 
	To this end, we let $\varphi(t,x) = \eta(t) \psi(x)$ such that $\psi \in C^{0,1}(\Ox)$ is compactly supported in $S \times [0,1]$ and $\eta \in C_c^\infty(0,T')$. After point evaluation on the grid points $x \in \tilde{\Lambra}_n$ and subsequent interpolation we get from \eqref{dyn_weak_sltn} for large enough $n$
		\begin{align}\label{eq:ODE-with-J}
			0 = \int_{0}^{T'} \int_{\Ox_{\nu_n}} \partial_t \bar y_n \cdot \partial_t \bar \varphi \ dx \ dt 
			-\int_{0}^{T'} \int_\Ox R_n \tilde J^n \colon \bar{\nabla}_n \bar \varphi \ dx \ dt 
			+ h_n \int_0^{T'} \int_{\Ox_{\nu_n}} \bar g \bar \varphi_3. 
		\end{align} 
	
	\begin{proposition}\label{prop: orth-tilde-J}
		The map $\tilde J$ satisfies
		\[ \sum_{l=1}^{4} \tilde J_{\cdot l} = \sum_{l=5}^{8} \tilde{J}_{\cdot l}. \]
		In particular $\tilde J(t,x) \perp (\R^3 \otimes Z) e_3$ for almost every $(t,x) \in I_T \times \Ox$.
	\end{proposition}

	\begin{proof}
        It suffices to check these assertions on $I_T \times S \times (0,1)$ since $\tilde J$ vanishes elsewhere.
                
		Suppose first $\nu_n \to \infty$. Let $T' \in I_T$ and $\varphi(t,x) = \eta(t) \int_0^{x_3} \Psi(x',\xi_3) \, d\xi_3$ with $\eta \in C_c^\infty(0,T')$ and $\Psi \in C^\infty(\Ox; \R^3)$ compactly supported in $S \times [0,1]$. Then
		\begin{equation}
			h_n \bar{\nabla}_n \bar \varphi \to \left(- \frac{\partial_3 \varphi}{2} , \ldots, \frac{\partial_3 \varphi}{2} , \ldots \right) \label{hngradient}
		\end{equation}
		uniformly and therefore also in $L^1(0,T'; L^2(\Ox; \R^3))$. 
		
		Inserting $\varphi$ into \eqref{eq:ODE-with-J} and, after multiplication by $h_n$, passing to the limit $n \to \infty$ in that equation, with the help of \eqref{bound-derivative-td} and \eqref{hngradient} we get 
		\begin{align*}
			0 =\underset{n \to \infty}{\longrightarrow} - \frac{1}{2} \int_0^{T'} \int_\Ox J \colon \left(- \partial_3 \varphi, \ldots, \partial_3 \varphi, \ldots \right) \ dx \ dt.
		\end{align*}
		By density we obtain for every $\chi \in L^1(0,T';L^2(S \times (0,1)))$
		\[ 0 = \int_0^{T' } \int_{S \times (0,1)} J \colon \left( -\chi, \ldots, \chi \ldots \right) \ dx \ dt, \]
	from which the claim follws. 
	
	Now suppose $\nu_n \equiv \nu \in \{2, 3, \ldots\}$. 
       Let $l \in \{1,\ldots,\nu-1\}$ and suppose $\eta \in C_c^\infty(0,T')$, $\chi \in C_c^\infty(S; \R^3)$ and $\phi(t,x') = \eta(t) \chi(x')$. 
        Set $\phi_0 = \ldots = \phi_{l-1} = \phi$ and $\phi_l = \ldots = \phi_{\nu-1} = 0$. 
        For $s \in [\frac{m-1}{\nu - 1} , \frac{m}{\nu - 1})$, $m = 1, \ldots, \nu - 1$, we interpolate linearly between the layers, i.e. we set
		\begin{align*}
			\varphi(t,x',s) 
			&= \big( m - (\nu - 1) s \big) \phi_{m-1}(t,x') 
			+ \big( (\nu - 1)s  - (m-1) \big) \phi_m(t,x').
		\end{align*}
		Then
		\begin{align}
			h_n \bar{\nabla}_n \varphi\Big(t,x', \frac{2m - 1}{2(\nu - 1)}\Big) 
			\underset{n \to \infty}{\longrightarrow}
			\frac{\nu - 1}{2}  \big( \phi(t,x'), \ldots, - \phi(t,x'), \ldots \big) \, \delta_{lm} \label{hngradient-fl}
		\end{align}
		uniformly. Again we insert $\varphi$ into \eqref{eq:ODE-with-J} and, after multiplying by $h_n$, pass to the limit $n \to \infty$ so as to get 
		\begin{align*}
			0 =
			-\frac{\nu - 1}{2} \int_0^{T'} \int_{{ \R^2} \times (\frac{l-1}{\nu-1}, \frac{l}{\nu-1})} \tilde{J}\colon (\phi, \ldots, -\phi, \ldots) \ dx \ dt
		\end{align*}
		with the help of \eqref{bound-derivative-td} and \eqref{hngradient-fl}, and hence indeed for every $\chi \in L^1(0,T';L^2(S))$
		\[ 0 = \int_0^{T'} \int_S \tilde{J}\Big(t,x', \frac{2l - 1}{2(\nu - 1)}\Big) \colon \left( -\chi(t,x'), \ldots, \chi(t,x') \ldots \right) \ dx' \ dt,\]
        which implies the claim since $\tilde{J}$ is piecewise constant in $x_3$. 
    \end{proof}
	
	With all these preparations we can now identify the zeroth and first moment in $x_3$ of the limiting strain $\tilde{J}$.  
	\begin{corollary}\label{cor:summary-tilde-J}
		The vertical average and first moment of $\tilde J$ satisfy: 
		\begin{enumerate}
			\item[a)] If $\nu_n \to \infty$, then 
			\[ \int_0^1 \tilde J \ dx_3 = \frac{1}{2} DQ_{\cell}^{\rel} \big(\hat{G}_1^{\sym} Z \big) \]
			and
			\[ \int_0^1 \Big(x_3 - \frac{1}{2}\Big) \tilde J \ dx_3 = 
				\frac{1}{24} DQ_{\cell}^{\rel} \big( \hat{G}_2 Z \big), \]
            almost everywhere on $(0,T') \times \R^2$, 
            \item[b)] if $\nu_n \equiv \nu \in \N$, then with $I_1 = (0, \frac{1}{\nu-1})$ and $I_2 = (\frac{\nu-2}{\nu-1},1)$ 
			\begin{align*}
				\int_0^1 J \ dx_3 
				&= \frac{1}{2} DQ_{\cell}^{\rel} \Big(\hat{G}_1^{\sym} Z + \frac{1}{2(\nu-1)} ( \hat{G}_2 Z_- +\partial_{12} v \, M) \Big), \\ 
				\sum_{j=1}^{2} \int_{I_j} J^{(j)} \ dx_3 
				&= \frac{1}{\nu - 1}  DQ_{\surf} \Big( \hat{G}_1^{\sym} Z^{(1)} + \frac{1}{2(\nu-1)} \partial_{12} v \, M^{(1)} \Big) 
            \end{align*}
			and 
			\begin{align*}
			 \int_0^1 \Big(x_3 - \frac{1}{2}\Big) J \ dx_3 
			 &= \frac{\nu(\nu-2)}{24(\nu-1)^2} DQ_{\cell}^{\rel} \big( \hat{G}_2 Z \big) , \\ 
			 \sum_{j=1}^{2} \int_{I_j} \Big( x_3 - \frac{1}{2} + \frac{(-1)^j}{2(\nu-1)} \Big) J^{(j)} \ dx_3
				&= \frac{1}{4(\nu-1)} DQ_{\surf} \big(  \hat{G}_2 Z^{(1)} \big)
			\end{align*}
        almost everywhere in  $(0,T') \times \R^2$. 
		\end{enumerate}
	\end{corollary} 

	\begin{proof}
        a) This follows from \eqref{conv-Jn}, \eqref{conv-Jnsurf}, Proposition~\ref{prop: strain-td}, \eqref{eq:G1G2}, Lemma~\ref{lem: derivative-relaxed},  Lemma~\ref{lemma: surf-to-bulk-orth}, Proposition~\ref{prop: orth-tilde-J}, and \eqref{eq: derivative-skew}. 
        Note that in the first integral the terms depending on $x_3$ integrate to $0$, while in the second integral the other terms vanish. 
        \smallskip 
        
        b) The same reasoning leads to the first two equations in (b) as the $x_3$ dependent terms integrate to $0$. For the second equation we also use that $\hat{G}_2 Z_-^{(1)} = - \hat{G}_2 Z_-^{(2)}$, so that these terms cancel in $D Q_{\surf}$.  
        
        For the third identity in (b) we see again that the terms in $J$ without explicit $x_3$ dependence integrate to $0$. Since $J$ is piecewise constant in $x_3$, we are thus left with 
        \begin{align*}
         \int_0^1 \Big(x_3 - \frac{1}{2}\Big) J \ dx_3
         &= \sum_{m=1}^{\nu-1} \frac{(2m-\nu)^2}{2(\nu-1)^2} \cdot \frac12 DQ_{\cell}^{\rel} ( \hat{G}_2 Z ) 
         = \frac{\nu(\nu-2)}{24(\nu-1)^2} DQ_{\cell}^{\rel} ( \hat{G}_2 Z ).
        \end{align*} 
        For the last identity we first obtain 
        \begin{align*}
            \MoveEqLeft
            \sum_{j=1}^{2} \int_{I_j} \Big( x_3 - \frac{1}{2} + \frac{(-1)^j}{2(\nu-1)} \Big) J^{(j)} \ dx_3\\
			&= \frac{1}{\nu-1} \sum_{j=1}^2 (-1)^j\frac{1}{2} \cdot \frac{1}{2} DQ_{\surf} \Big( \hat{G}_1^{\sym} Z^{(j)} + (-1)^j\frac{\nu-2}{2(\nu-1)} \hat{G}_2 Z^{(j)} \\ 
			& \qquad \qquad \qquad \qquad \qquad \qquad \qquad \qquad + \frac{1}{2(\nu-1)} ( \hat{G}_2 Z_-^{(j)} +\partial_{12} v M^{(j)} ) \Big) \\ 
			&= \sum_{j=1}^2 \frac{(-1)^j}{4(\nu-1)} DQ_{\surf} \Big( (-1)^j\frac{\nu-2}{2(\nu-1)} \hat{G}_2 Z^{(1)} 
			+ (-1)^j \frac{1}{2(\nu-1)} \hat{G}_2 Z^{(1)} \Big), 
        \end{align*}
        from which the claim follows. 
   
	\end{proof}

	We have the following symmetry properties.
	\begin{lemma}\label{lem:symmetric-properties-time-dependent}
		The scaled stresses $J^n$, $J^{(1,n)}$ and $J^{(2,n)}$ satisfy
		\begin{align}
			\| J^n Z^T - Z (J^n)^T \|_{L^\infty(0,T';L^1(\Ox;\R^{3 \times 3}))} 
			&\leq C(T')h_n^2, \\
			\| J^{(j,n)}(Z^{(j)})^T - Z^{(j)} (J^{(j,n)})^T)\|_{L^\infty(0,T';L^1(\Ox;\R^{3 \times 3}))} 
			&\leq C(T')h_n^2,\quad j=1,2.
		\end{align}
	\end{lemma}
	
	\begin{proof}
		We only show the first inequality. The other can be shown exactly the same way. 
		By \eqref{eq:W-sym-prop} $DW_{\cell}(F)F^T$ is symmetric for each $F \in \R^{3 \times 8}$, so 
		\begin{align*}
			0 
			&= DW_{\cell}(Z + h_n^2 \bar{G}_n) (Z + h_n^2 \bar{G}_n)^T - (Z + h_n^2 \bar{G}_n) DW_{\cell}(Z + h_n^2 \bar{G}_n)^T \\
			&= h_n^2 J^n Z^T + h_n^4 J^n \left(\bar{G}_n\right)^T - h_n^2 Z \left(J^n\right)^T - h_n^4 \bar{G}_n \left( J^n \right)^T. 
		\end{align*} 
		Hence, 
		\[ \big|J^n Z^T - Z \left(J^n\right)^T \big| 
		\leq h_n^2 \big|J^n \bar{G}_n^T - \bar{G}_n \left(J^n\right)^T\big|. \]
		Integrating and applying the Cauchy Schwarz inequality to the latter term yields the claim. 
	\end{proof}

\section{Proof of the main result}\label{section:Proofs}

\subsection{Outline of the proof}	
	In {\it Step 1} we will derive an energy bound which will be used to show that the solutions to \eqref{ode} exist up to time $T$. {\it Step 2} deals with the convergence of the displacements. In {\it Step 3} we will show that the equation \eqref{dyn_weakform2} holds true. {\it Step 5} is to show the weak continuity statements as well as that $v$ satisfies the stated initial conditions. {\it Step 4}, the derivation of equation \eqref{dyn_weakform1}, is the most complex task. Thus we split it into three parts. Let $\phi(t,x') = \mu(t) \chi(x')$ with $\mu \in C_c^\infty(0,T')$ and $\chi \in C_c^\infty(S)$. Set $\varphi(t,x) = (0,0,\phi(t,x'))$. In {\bf Part 1} the goal is to show
	\begin{align}
	      \MoveEqLeft
	    \int_0^{T'} \int_\Ox h_n^{-1} \tilde J^n \colon \bar{\nabla}_n \bar \varphi \ dx \ dt \nonumber \\
		&\underset{n \to \infty}{\longrightarrow }\int_0^{T'} \int_S \Big( \partial_t v \, \partial_t \phi 
		- \frac{1}{2}DQ_{\cell}^{\rel} ( \hat{G}_1^{\sym} Z) \colon (\nabla' v \otimes \nabla' \phi)\widehat{\phantom{i}} \, Z + g \phi \Big) \ dx' \ dt \label{timedep_ELE1_1}
	\end{align}
	for $\nu_n \to \infty$ and
	\begin{align}
		\MoveEqLeft
		\int_0^{T'} \int_\Ox h_n^{-1} \tilde J^n \colon \bar{\nabla}_n \bar \varphi \ dx \ dt \nonumber \\
		&\underset{n \to \infty}{\longrightarrow }
		\frac{\nu}{\nu - 1} \int_0^{T'} \int_S \partial_t v  \, \partial_t \phi \ dx' \ dt \nonumber \\
        &\qquad- \int_0^{T'} \int_S \frac{1}{2} DQ_{\cell}^{\rel} \left(\hat{G}_1^{\sym} Z + \frac{1}{2(\nu-1)} \hat{G}_2 Z_- +\frac{\partial_{12} v}{2(\nu-1)}  \, M\right) \colon 
			(\nabla' v \otimes \nabla' \phi)\widehat{\phantom{i}} \, Z 
			\ dx' \ dt \nonumber \\
        &\qquad- \int_0^{T'} \int_S \frac{1}{\nu - 1} DQ_{\surf} \Big( \hat{G}_1^{\sym} Z^{(1)} + \frac{\partial_{12} v}{2(\nu-1)}  \, M^{(1)} \Big) \colon 
			(\nabla' v \otimes \nabla' \phi)\widehat{\phantom{i}} \, Z^{(1)} 
			\ dx' \ dt \nonumber \\
		&\qquad+ \frac{\nu}{\nu-1}  \int_0^{T'} \int_S g \, \phi  \ dx' \ dt \label{timedep_ELE1_3}
	\end{align}
	for $\nu_n \equiv \nu \in \N$.
	
	In {\bf Part 2} we will show that for $\nu_n \to \infty$, where $\bar{x}$ is such that $\bar{Q}_n(x) = \bar{x} + (- \frac{\varepsilon_n}{2}, \frac{\varepsilon_n}{2} )^2 \times (-\frac{\varepsilon_n}{2 h_n}, \frac{\varepsilon_n}{2 h_n}) $ and $(\nabla' \phi)\dtilde{\phantom{\iota}}(t,\bar{x}) = \frac{1}{8}\sum_{j=1}^8\nabla' \phi(t,\bar{x}' + (z^j)')$, 
	\begin{align}
		\MoveEqLeft
		\sum_{l=1}^{8} \int_0^{T'} \int_\Ox h_n^{-1} (R_n \tilde J^n)_{\cdot l} \cdot  z_3^l \big((\nabla' \phi)\dtilde{\phantom{\iota}}(t,\bar{x})^T, 0 \big)^T \ dx \ dt \nonumber \\
		&\underset{n \to \infty}{\longrightarrow}
		\int_0^{T'} \int_S \frac{1}{24} DQ_{\cell}^{\rel} ( (\nabla'^2 v)\widehat{\phantom{i}} \,Z ) \colon (\nabla'^2 \phi) \widehat{\phantom{i}} \,Z \ dx' \ dt. \label{timedep_ELE1_5}
	\end{align}
	and for $\nu_n \equiv \nu \in \N$
	\begin{align}
		\MoveEqLeft
		\sum_{l=1}^{8} \int_0^{T'} \int_\Ox h_n^{-1} (R_n \tilde J^n)_{\cdot l} \cdot  z_3^l \big((\nabla' \phi)\dtilde{\phantom{\iota}}(t,\bar{x})^T, 0 \big)^T \ dx \ dt \nonumber \\
		&\underset{n \to \infty}{\longrightarrow}
		- \int_0^{T'} \int_S \frac{\nu (\nu - 2)}{24(\nu - 1)^2} DQ_{\cell}^{\rel} (\hat{G}_2 Z) \colon (\nabla'^2 \phi )\widehat{\phantom{i}} \, Z \ dx' \ dt \nonumber \\
		& - \int_0^{T'} \int_S \frac{1}{4(\nu - 1)} DQ_{\cell}^{\rel} \Big(\hat{G}_1^{\sym} Z + \frac{1}{2(\nu-1)} \hat{G}_2 Z_- +\frac{\partial_{12} v}{2(\nu-1)}  \, M\Big) \colon 
			(\nabla'^2 \phi)\widehat{\phantom{i}} \, Z_-
			\ dx' \ dt \nonumber \\
		&- \int_0^{T'} \int_S \frac{1}{4(\nu - 1)} DQ_{\surf} ( \hat{G}_2 Z^{(1)} ) \colon (\nabla'^2\phi)\widehat{\phantom{i}} Z^{(1)} \ dx' \ dt. \label{timedep_ELE1_5-nu}
	\end{align}
	
	In {\bf Part 3} we determine the convergence of
	\begin{equation}
		\int_0^{T'} \int_\Ox h_n^{-1} \tilde J^n \colon \bar{\nabla}_n \bar \varphi \ dx \ dt
		-\sum_{l=1}^{8} \int_0^{T'} \int_\Ox h_n^{-1} (R_n \tilde J^n)_{\cdot l} \cdot  z_3^l \big( (\nabla' \phi)\dtilde{\phantom{\iota}}(t,\bar{x})^T, 0 \big)^T \ dx \ dt.  \label{term1}
	\end{equation}
	For $\nu_n \to \infty$ we will prove that $\eqref{term1} \to 0$ which yields \eqref{dyn_weakform1} for $\nu_n \to \infty$. For $\nu_n \equiv \nu \in \N$ we have
	\begin{align}
		\eqref{term1} 
		&\to
		\int_0^{T'} \int_S \frac{1}{4(\nu - 1)} DQ_{\cell}^{\rel} \Big(\hat{G}_1^{\sym} Z + \frac{1}{2(\nu-1)} \hat{G}_2 Z_- +\frac{\partial_{12} v}{2(\nu-1)}  \, M\Big) \colon \partial_{12} \phi M \ dx' \ dt \nonumber \\
		&\quad + \int_0^{T'} \int_S \frac{1}{2(\nu - 1)^2} DQ_{\surf} \Big( \hat{G}_1^{\sym} Z^{(1)} + \frac{\partial_{12} v}{2(\nu -1)} M^{(1)}  \Big) \colon \partial_{12} \phi M^{(1)} \ dx' \ dt, \label{endpart3}
	\end{align}
	which implies \eqref{dyn_weakform1} for $\nu_n \equiv \nu$. 


\subsection{Proof of the main theorem}
	We will follow the structure given in the outline. At this point we remark that solutions $y_n$ of \eqref{ode} in particular satisfy \eqref{dyn_weak_sltn}.
	\begin{proof}[Proof of Theorem~\ref{thm: main-td}]
		
		\textit{Step 1: An energy bound and existence time of the solutions.} From the bound on the initial data we will show the inequalities
		\begin{align}
			\esssup_{t \in [0,T']} \frac{\varepsilon_n^3}{h_n} \sum_{x \in \tilde{\Lambra}_n^{\prime} } W \left(x, \bar{\nabla}_n y_n(t,x)\right) \leq C(T') h_n^4, \label{energy1} \\
			\esssup_{t \in [0,T']} \frac{\varepsilon_n^3}{h_n} \sum_{x \in \tilde{\Lambra}_n } \abs{\partial_t y_n(t,x)}^2 \leq C(T') h_n^2 \label{energy2}
		\end{align}
		for every $T' \in I_T$.
		We may assume for a moment that $T_n \leq T'$. Now we show the energy bound which will eventually yield that $T_n \geq T'$ must hold for every $n \in \N$.
		Since $\partial_t y_n(t,x)=0$ for $x \in \tilde{\Lambra}_n \setminus \tilde{\Lambda}_{n,S}$, we may insert $\varphi = \partial_t y_n$ in \eqref{dyn_weak_sltn} and deduce
		\[ \frac{d}{dt} \left( \frac{h_n^2}{2} \sum_{x \in \tilde{\Lambra}_n} \abs{\partial_t y_n(t,x)}^2 + \sum_{x \in \tilde{\Lambra}^\prime_n} W(x, \bar{\nabla}_n y_n (t, x)) \right) = h_n^3 \sum_{x \in \tilde{\Lambra}_n} g(t,x') \left(\partial_t y_n (t,x) \right)_3. \]
		Thus we get 
		\begin{align}
            \MoveEqLeft
			\frac{h_n^2}{2} \sum_{x \in \tilde{\Lambra}_n} \abs{\partial_t y_n(t,x)}^2 + \sum_{x \in \tilde{\Lambra}_n^\prime} W \left(x, \bar{\nabla}_n y_n(t,x) \right) \nonumber \\
			&= \frac{h_n^2}{2} \sum_{x \in \tilde{\Lambra}_n} \abs{\partial_t y_n(0,x)}^2 + \sum_{x \in \tilde{\Lambra}_n^\prime} W(x, \bar{\nabla}_n y_n(0,x)) \nonumber \\
			&\qquad + \int_0^{t} h_n^3 \sum_{x \in \tilde{\Lambra}_n} g(s,x') \left( \partial_t y_n(s,x) \right)_3 \ ds \label{energy_ineq}
		\end{align}
		for every $0 \leq t \leq T_n$.
		From \eqref{energy_ineq2} and \eqref{energy_ineq} we deduce using Young's inequality in the form $\abs{ab} \leq h \frac{a^2}{2} + \frac{b^2}{2h}$
		\begin{align}
            \MoveEqLeft
			\frac{h_n^2}{2} \frac{\varepsilon_n^3}{h_n} \sum_{x \in \tilde{\Lambra}_n } \abs{\partial_t y_n(t,x) }^2 + \frac{\varepsilon_n^3}{h_n} \sum_{x \in \tilde{\Lambra}_n^{\prime} } W\left(x,\bar{\nabla}_n y_n(t,x) \right) \nonumber \\
			&= Ch_n^4 + \frac{h_n^4}{2} \int_0^{t} \int_{\Ox_{\nu_n}} \abs{\bar g(s,x')}^2 \ dx \ ds
			+ \frac{h_n^2}{2} \int_0^{t} \frac{\varepsilon_n^3}{h_n} \sum_{x \in \tilde{\Lambra}_n} \abs{\partial_t y_n(s,x) }^2 \ ds \nonumber \\
			&\leq C h_n^4 + \frac{h_n^2}{2} \int_0^{t} \frac{\varepsilon_n^3}{h_n} \sum_{x \in \tilde{\Lambra}_n} \abs{\partial_t y_n(s,x) }^2 \ ds. \label{energyhelp}
		\end{align}
		Thus
		\[\frac{\varepsilon_n^3}{h_n} \sum_{x \in \tilde{\Lambra}_n} \abs{\partial_t y_n(t,x) }^2
		\leq C h_n^2 + \int_0^{t} \frac{\varepsilon_n^3}{h_n} \sum_{x \in \tilde{\Lambra}_n} \abs{\partial_t y_n(s,x) }^2 \ ds\]
		and applying Gronwall's inequality yields
		\[ \frac{\varepsilon_n}{h_n} \sum_{x \in \tilde{\Lambra}_n} \abs{\partial_t y_n(t,x)}^2 \leq C h_n^2 + C h_n^2 \exp(T_n) 
		\leq C h_n^2 (1 + \exp(T'))  = C(T') h_n^2. \]
		Integration leads to
		\[\int_0^{T_n} \frac{\varepsilon_n^3}{h_n} \sum_{x \in \tilde{\Lambra}_n} \abs{\partial_t y_n(t,x) }^2 \ dt \leq C(T') h_n^2. \] 
		Together with \eqref{energyhelp} this immediately implies the inequalities \eqref{energy1} and \eqref{energy2} up to time $T_n$.
		By \eqref{eq: interpol-bounds} we get the same energy bound for the piecewise affine interpolations, i.e., 
		\begin{equation}
			\esssup_{t \in [0,T_n]} \int_\Ox \abs{ \partial_t {\tilde{y}}_n(x) }^2 \ d x \leq C(T_n) h_n^2. \label{bound-dt-tildeyn}
		\end{equation}
		
		\bigskip
		\textit{}
		
		Evidently \eqref{energy2} rules out a finite time blow-up for $\partial_t y_n$, i.e. for every $n \in \N$ there is a constant $C(n)$ such that
		\begin{equation}
			\esssup_{t \in [0,T']} \norm{ \partial_t y_n(t, \cdot) }_{l^\infty(\tilde{\Lambra}_n)} \leq C(n). \label{bound-dtyn}
		\end{equation}
		It follows from \eqref{energy1} and the lower bound of $W$ that
		\begin{equation}
			\esssup_{t \in [0,T_n]} \norm{\bar{\nabla}_n y_n(t,\cdot)}_{l^\infty(\tilde{\Lambra}_n^\prime)} \leq C(n). \label{derivative-bound}
		\end{equation} 
		Now if $x \in \tilde{\Lambra}_n$ we can choose $x_0 \in \tilde{\Lambra}_n \setminus \tilde{\Lambda}_{n,S}$ such that $x = x_0 + l e_1$, with $l$ only depending on $n$. Then \eqref{derivative-bound} and $y(x_0) = (x_0', h (x_0)_3)$ imply  
		\[ y_n(x) 
		   = (x_0', h (x_0)_3) + \sum_{j=1}^l \big( \bar\partial_2^n y(x_j) - \bar\partial_1^n y(x_j) \big), \] 
        where $x_j = x_0 + (j-1) e_1 - (\varepsilon_n (z^1)^\prime, \frac{\varepsilon_n}{h_n} z_3^1)$. This proves 
		\begin{equation*}
			\esssup_{t \in [0,T_n]} \norm{y_n(t,\cdot)}_{l^\infty(\tilde{\Lambra}_n)} \leq C(n). 
		\end{equation*}
		Together with \eqref{bound-dtyn} this implies $T_n = T$.
		\bigskip

		\textit{Step 2: Convergence of the displacements and boundary conditions}
		The convergence of the displacements $u_n$ and $v_n$ as well as the boundary conditions \eqref{bc-td-limiting} are contained in Proposition~\ref{prop: td-displacements}.
		
		\bigskip
		\textit{Step 3: Derivation of equation \eqref{dyn_weakform2}.}
		Let $\Psi(t,x') = \eta(t) \chi(x')$ with $\eta \in C_c^\infty(0,T')$ and $\chi \in C_c^\infty(S;\R^2)$. With $\varphi(t,x) \coloneqq (\Psi(t,x'), 0)$ we have by \eqref{dyn_weak_sltn}
		\begin{align*}
			0 &= h_n^2 \int_0^{T'} \int_{ \Ox_{\nu_n} }  \partial_t \bar y_n(t, x) \cdot \partial_t \bar \varphi (t, x) \ d x \ dt \\
			&\qquad - \int_0^{T'} \int_\Ox D_A W \left(x, \bar{\nabla}_n \bar y_n (t, x)  \right) \colon \bar{\nabla}_n \bar \varphi (t, x) \ d x \ dt.
		\end{align*}
		Note that the force term does not appear here because of $\varphi_3 = 0$. After multiplication with $h_n^{-2}$ and decomposition of the discrete gradient we get
		\begin{align}
			0 =  \int_0^{T'} \int_{ \Ox_{\nu_n}}  \partial_t \bar y_n(t, x) \cdot \partial_t \bar \varphi (t, x) \ d x \ dt 
			- \int_0^{T'} \int_\Ox \tilde J^n \colon \bar{\nabla}_n \bar \varphi \ dx \ dt.  \label{td-bulk-and-surf}
		\end{align} 
		The first term on the right hand side of \eqref{td-bulk-and-surf} vanishes as $n \to \infty$ because of \eqref{energy2}. For the seond term we note that $\tilde J^n \overset{*}{\rightharpoonup} \tilde J$ in $L^\infty_{\mathrm{loc}}(I_T; L^2(\Ox; \R^{3 \times 8}))$ and $\bar{\nabla}_n \doublebar \varphi \to \nabla \varphi Z$ in $L^1_{\mathrm{loc}}(I_T;L^2(\Ox; \R^{3 \times 8}))$, hence we get
		\[ - \int_0^{T'} \int_\Ox \tilde J \colon \nabla \varphi Z \ d x \ dt = 0. \] 
		
		For $\nu_n \to \infty$ we obtain by part (a) of Corollary~\ref{cor:summary-tilde-J} 
		\begin{align*}
			0 = \int_0^{T'} \int_S \frac{1}{2} DQ_{\cell}^{\rel} ( \hat{G}_1^{\sym} \,Z ) \colon (\nabla' \Psi) \widehat{\phantom{i}} \,Z \ dx' \ dt.
		\end{align*}
		By density this holds for every $\Psi \in L^2(0,T';H_0^1(S;\R^2))$ and we have \eqref{dyn_weakform2} for $\nu_n \to \infty$.  
		
		For $\nu_n \equiv \nu \in \N$ by part (b) of Corollary~\ref{cor:summary-tilde-J}
		\begin{align*}
            0 
			&=\int_0^{T'} \int_S \frac{1}{2} DQ_{\cell}^{\rel} \left(\hat{G}_1^{\sym} Z + \frac{1}{2(\nu-1)} \hat{G}_2 Z_- +\frac{1}{2(\nu-1)} \partial_{12} v \, M\right) \colon 
			(\nabla' \Psi)\widehat{\phantom{i}} \, Z  \ dx' \ dt \\
			&\qquad + \frac{1}{\nu - 1}  \int_0^{T'} \int_S DQ_{\surf} \left( \hat{G}_1^{\sym} Z^{(1)} + \frac{1}{2(\nu-1)} \partial_{12} v \, M^{(1)} \right) \colon (\nabla' \Psi)\widehat{\phantom{i}} \, Z^{(1)} \ dx' \ dt  
		\end{align*}

		for every $\Psi \in L^2(0,T'; H_0^1(S))$, which is \eqref{dyn_weakform2} for $\nu_n \equiv \nu$.
		\bigskip
		
		\textit{Step 4: Derivation of equation \eqref{dyn_weakform1}.}

		{\bf Part 1:} Let $\phi(t,x') = \mu(t) \chi(x')$ with $\mu \in C_c^\infty(0,T')$ and $\chi \in C_c^\infty(S)$. Inserting $\varphi_1(t,x) = (0, 0, \phi(t,x'))$ into \eqref{eq:ODE-with-J}, multiplying by $h_n^{-1}$ and setting $A_n \coloneqq \frac{R_n - Id}{h_n}$ we obtain 
		Letting $A_n \coloneqq \frac{R_n - Id}{h_n}$ and multiplying with $h_n^{-1}$ leads to
		\begin{align}
			\MoveEqLeft
			\int_0^{T'} \int_\Ox h_n^{-1} \tilde J^n \colon \bar{\nabla}_n \bar \varphi_1 \ dx \ dt \nonumber \\
			&= \int_0^{T'} \int_{\Ox_{\nu_n}}  h_n^{-1} \partial_t \left(\bar y_n \right)_3  \partial_t \bar \phi \ dx \ dt \label{ELE1_td_1} \\
			&\qquad - \int_0^{T'} \int_\Ox A_n \tilde J^n \colon \bar{\nabla}_n \bar \varphi_1 \ dx \ dt \label{ELE1_td_2} \\
			&\qquad + \frac{\nu_n}{\nu_n - 1} \int_0^{T'} \int_{\Ox_{\nu_n}} \bar g \bar \phi \ dx' \ dt \label{ELE1_td_3}.
		\end{align}
		We determine the convergence of these terms separately. First look at \eqref{ELE1_td_3}: 
		Since $g \in L^2(I_T; W^{1,\infty}(S)) \cap C^0(I_T, L^\infty(S))$, we see that $\bar g \bar \phi \to g \phi$ in $L^2(0, T'; L^\infty(S))$ so that 
		\begin{align*}
			\eqref{ELE1_td_3} \underset{n \to \infty}{\longrightarrow}
			\begin{cases}
				\int_0^{T'} \int_S g \phi \ dx' \ dt & \text{if} \ \nu_n \to \infty, \\
				\frac{\nu}{\nu - 1} \int_0^{T'} \int_S g \phi \ dx' \ dt & \text{if} \ \nu_n \equiv \nu \in \N. 
			\end{cases}
		\end{align*} 
		For the term \eqref{ELE1_td_2} we note that $(\bar{\nabla}_n \varphi_1)_{ij} = 0$ whenever $i \in \{1,2\}$ and $(\bar{\nabla}_n \varphi_1)_{3j} = \bar\partial_j^n \phi$ for $j=1,\ldots, 8$ and therefore
		\begin{align*}
			A_n \tilde J^n \colon \bar{\nabla}_n \varphi_1(t,x) 
			&= \sum_{l=1}^{8} ( A_n \tilde J^n )_{3l} \bar\partial_l^n \phi(t,x') \\
			&= \sum_{l=1}^{8} \sum_{k=1}^{2} (A_n)_{3k} \tilde J^n_{kl} \bar\partial_l^n \phi(t,x') 
			+ \sum_{l=1}^{8} (A_n)_{33} \tilde J^n_{3l} \bar\partial_l^n \phi(t,x').
		\end{align*}
		By Proposition~\ref{prop: td-displacements-further} we have that $A_n e_i \to Ae_i$ for $A = e_3 \otimes \nabla' v - \nabla' v \otimes e_3 $ strongly in $L^q_{\mathrm{loc}}(I_T;L^p(\R^2, \R^3))$ for $i = 1,2$ and any $1 \leq p,q < \infty$ as well as
		\[ \sym \ A_n \to 0 \quad \text{strongly in} \  L^\infty_{\mathrm{loc}} (I_T;L^r(\R^2, \R^{3 \times 3})) \]
		for all $r < \infty$. In particular this implies that $(A_n)_{33} \to 0$ strongly in $L^\infty_{\mathrm{loc}}(I_T; L^r(\R^2))$
		and
		\begin{align*}
			(\ref{ELE1_td_2}) \underset{n \to \infty} {\longrightarrow} - \int_0^{T'} \int_\Ox A \tilde J \colon  \nabla \varphi_1 Z \ dx \ dt 
			= - \int_0^{T'} \int_\Ox \tilde J \colon(\nabla' v \otimes \nabla' \phi)\widehat{\phantom{i}} \, Z \ dx \ dt,
		\end{align*}
		where we have used that $A^T \nabla \varphi_1 = (\nabla' v \otimes \nabla' \phi)\widehat{\phantom{i}}$. 
		
		For the term \eqref{ELE1_td_1} we note that \eqref{vn_3d_strong} and Proposition~\ref{prop: equiv-convergences} imply that 
		\begin{align}\label{eq: vn-convergence} h_n^{-1}  (\bar y_n)_3 \to 
		\begin{cases}
			\hat{v} \ \text{ in }  L^\infty(0,T';L^2_{ \mathrm{loc}}(\Oxc)) & \text{if} \ \nu_n \to \infty, \\
			\hat{v}^\ast \text{ in } L^\infty(0,T';L^2_{ \mathrm{loc}}(\Oxc_{\nu}))  & \text{if} \ \nu_n \equiv \nu \in \N,
		\end{cases}
		\end{align}
		where $\hat{v}(t,x) = v(t,x') + \left(x_3-\frac{1}{2}\right)$ and $\hat v^\ast(t,x',x_3) = \hat v(t,x',\frac{i}{\nu - 1})$ if $x_3 \in (\frac{2i-1}{2(\nu - 1)}, \frac{2i+1}{2(\nu-1)} )$, $i = 1, \ldots, \nu - 1$. Moreover, by \eqref{energy2}, $h_n^{-1}  (\partial_t \bar y_n)_3$ is bounded in $L^\infty_{\mathrm{loc}}(I_T; L^2(\Ox))$, respectively, $L^\infty_{\mathrm{loc}}(I_T; L^2(\Ox_{\nu}))$, and therefore $h_n^{-1}  (\partial_t \bar y_n)_3 \overset{\ast}{\rightharpoonup} \partial_t \hat{v}$ or $h_n^{-1}  (\partial_t \bar y_n)_3 \overset{\ast}{\rightharpoonup} \partial_t \hat{v}^\ast$ in the respective spaces. Together with $\partial_t \bar \phi \to \partial_t \phi$ in $L^\infty_{\mathrm{loc}}(0,T;L^2(S))$ we deduce that
		\[ \eqref{ELE1_td_1} \to 
		\begin{cases}
			\int_0^{T'} \int_S \partial_t v \, \partial_t \phi \ dx' \ dt & \text{if} \ \nu_n \to \infty, \\
			\frac{\nu}{\nu - 1} \int_0^{T'} \int_S \partial_t v \, \partial_t \phi \ dx' \ dt & \text{if} \ \nu_n \equiv \nu \in \N
		\end{cases} \]
		as $n \to \infty$.
		
		Combining the convergences of \eqref{ELE1_td_1}, \eqref{ELE1_td_2} and \eqref{ELE1_td_3}, we obtain \eqref{timedep_ELE1_1} via  by part (a) of Corollary~\ref{cor:summary-tilde-J} in case $\nu_n \to \infty$ and \eqref{timedep_ELE1_3} by part (b) of Corollary~\ref{cor:summary-tilde-J} in case $\nu_n \equiv \nu \in \N$. 
		\bigskip
		
		{\bf Part 2:} Let $\varphi_2(t,x) = (x_3 - \frac{1}{2})\big(\eta(t,x')^T, 0\big)^T$ with $\eta(t,x')^T = \big(\partial_1 \phi(t,x'), \partial_2 \phi(t,x')\big)$. In \eqref{eq:ODE-with-J} we then get 
		\begin{align}
			0 = \int_0^{T'} \int_{\Ox_{\nu_n}} \partial_t \bar y_n \cdot \partial_t \bar \varphi_2 \ dx \ dt 
			- \int_0^{T'} \int_\Ox R_n \tilde J^n \colon \bar{\nabla}_n \bar \varphi_2 \ dx \ dt.\label{ELE1_td_4} 
		\end{align}
		We decompose the second term on the right hand side with the help of the product rule \eqref{eq: bar-product-rule} by writing  
		\begin{align}\label{eq: prodrule-for-phi2}
		\bar{\partial}_l^n \bar \varphi_2(t,\bar{x}) 
		  = h_n^{-1} z_3^l \begin{pmatrix} \dtilde{\eta}(t,\bar{x}') \\ 0 \end{pmatrix} + \Big( \bar{x}_3 + \frac{\varepsilon_n}{h_n}z_3^l - \frac{1}{2} \Big) \begin{pmatrix}\bar{\partial}_l^n \bar\eta(t, \bar{x}') \\ 0 \end{pmatrix},
		\end{align}                      
        where we denote by $\bar{x} \in \tilde{\Lambra}_n$ the center of a lattice cell containing $x$, i.e., $x \in \tilde{Q}_n(\bar{x})$, we have written $\dtilde{\eta}(t,\bar{x}') = \frac{1}{8}\sum_{j=1}^8\eta(t,\bar{x}' + (z^j)')$ and we have used that 
        \[ \sum_{j=1}^{8} h_n^{-1} z_3^j \, \bar\eta(t,\bar{x}' + \varepsilon_n (z^j)') 
        = \bigg(-\frac12 + \frac12 \bigg)\sum_{j=1}^{4} h_n^{-1} z_3^j \, \bar\eta(t,\bar{x}' + \varepsilon_n (z^j)') = 0. \]
        So 
        \begin{align}
            \MoveEqLeft
			\sum_{l=1}^{8} \int_0^{T'} \int_\Ox \big(R_n \tilde J^n\big)_{\cdot l} \cdot h_n^{-1} z_3^l \begin{pmatrix} \dtilde{\eta}(t,\bar{x}') \\ 0 \end {pmatrix} \ dx \ dt \label{eq: part-ii-term} \\
			&= \int_0^{T'} \int_{\Ox_{\nu_n}}  \partial_t \bar y_n  \partial_t \bar \varphi_2 \ dx \ dt \label{ELE1_td_16}\\
			&\qquad - \sum_{l=1}^{8} \int_0^{T'} \int_\Ox \big(R_n \tilde J^n\big)_{\cdot l} \cdot  \Big( \bar{x}_3 + \frac{\varepsilon_n}{h_n}z_3^l - \frac{1}{2} \Big) \begin{pmatrix} \bar{\partial}_l^n \bar\eta(t, \bar{x}') \\ 0 \end {pmatrix} \ dx \ dt.   \label{ELE1_td_11}
		\end{align} 

		The term in \eqref{ELE1_td_16} tends to $0$ by the energy inequality \eqref{energy2}. For the term in \eqref{ELE1_td_11} we use that $R_n \tilde J^n \overset{\ast}{\rightharpoonup} \tilde J$ in $L^\infty_{\mathrm{loc}}(0,T;L^2(\Ox;\R^{3 \times 8}))$ and that 
        \[ \Big( \bar{x}_3 + \frac{\varepsilon_n}{h_n}z_3^l - \frac{1}{2} \Big) \begin{pmatrix} \bar{\partial}_l^n \bar \eta(t, \bar{x}') \\ 0 \end {pmatrix} 
        \to  
        \Big(x_3-\frac{1}{2}\Big) \begin{pmatrix} \nabla' \eta(t,x) (z^l)' \\ 0 \end{pmatrix} \]
        uniformly if $\nu_n \to \infty$ and 
        \[ \Big( \bar{x}_3 + \frac{\varepsilon_n}{h_n}z_3^l - \frac{1}{2} \Big) \begin{pmatrix} \bar{\partial}_l^n \bar\eta(t, \bar{x}') \\ 0 \end {pmatrix} 
        \to  
        \Big(\frac{2i-1}{2(\nu-1)}+\frac{1}{\nu-1}z_3^l -\frac{1}{2}\Big) \begin{pmatrix} \nabla' \eta(t,x) (z^l)' \\ 0 \end{pmatrix} \]
        for $x_3 \in (\frac{i-1}{\nu-1},\frac{i}{\nu-1})$ uniformly in case $\nu_n \equiv \nu \in \N$.         
        If $\nu_n \to \infty$, we thus obtain 
        \begin{align*}
         \eqref{eq: part-ii-term}
         &\underset{n \to \infty}{\longrightarrow}
         - \int_0^{T'} \int_\Ox  \Big( x_3 - \frac{1}{2} \Big) \tilde{J} \colon (\nabla'^2 \phi )\widehat{\phantom{i}} \, Z \ dx \ dt \\ 
         &= - \int_0^{T'} \int_S \frac{1}{24} DQ_{\cell}^{\rel} \big( \hat{G}_2 Z \big) \colon (\nabla'^2 \phi )\widehat{\phantom{i}} \, Z \ dx' \ dt
        \end{align*}
        by part (a) of Corollary~\ref{cor:summary-tilde-J}, which implies \eqref{timedep_ELE1_5}. If  $\nu_n \equiv \nu \in \N$ since $\tilde{J}$ is piecewise constant in $x_3$ we get by part (b) of Corollary~\ref{cor:summary-tilde-J} 
        \begin{align*}
         \eqref{eq: part-ii-term}
         &\underset{n \to \infty}{\longrightarrow}
         - \int_0^{T'} \int_\Ox  \Big( x_3 - \frac{1}{2} \Big) \tilde{J} \colon (\nabla'^2 \phi )\widehat{\phantom{i}} \, Z + \frac{1}{2(\nu-1)} \tilde{J} \colon (\nabla'^2 \phi )\widehat{\phantom{i}} \, Z_- \ dx \ dt \\ 
         &= - \int_0^{T'} \int_\Ox  \Big( x_3 - \frac{1}{2} \Big) J \colon (\nabla'^2 \phi )\widehat{\phantom{i}} \, Z + \frac{1}{2(\nu-1)} J \colon (\nabla'^2 \phi )\widehat{\phantom{i}} \, Z_- \ dx \ dt \\ 
         &\qquad - \sum_{j=1}^{2} \int_0^{T'} \int_{\R^2} \int_{I_j} \Big( x_3 - \frac{1}{2} + \frac{(-1)^j}{2(\nu-1)} \Big) J^{(j)} \colon (\nabla'^2 \phi )\widehat{\phantom{i}} \, Z^{(1)} \ dx \ dt \\ 
         &= - \int_0^{T'} \int_S \frac{\nu(\nu-2)}{24(\nu-1)^2} DQ_{\cell}^{\rel} ( \hat{G}_2 Z ) \colon (\nabla'^2 \phi )\widehat{\phantom{i}} \, Z_- \\ 
         &\qquad\qquad\quad + \frac{1}{4(\nu-1)} DQ_{\cell}^{\rel} \Big(\hat{G}_1^{\sym} Z + \frac{1}{2(\nu-1)} ( \hat{G}_2 Z_- +\partial_{12} v \, M) \Big) \colon (\nabla'^2 \phi )\widehat{\phantom{i}} \, Z_- \\
         &\qquad\qquad\quad + \frac{1}{4(\nu-1)}DQ_{\surf} ( \hat{G}_2 Z^{(1)} ) \colon (\nabla'^2 \phi )\widehat{\phantom{i}} \, Z^{(1)} \ dx \ dt, 
        \end{align*}
        which is \eqref{timedep_ELE1_5-nu}. 

        {\bf Part 3:} Finally we determine the limit of \eqref{term1} for $\varphi(t,x) = \varphi_1(t,x) = (0,0,\phi(t,x'))$. We write  $h_n^{-1} R_n = A_n + h_n^{-1} \Id$ as in the previous step and recall that $A_n e_i \to A e_i$ for $i= 1,2$ in $L^q_\mathrm{loc}(I_T; L^p(\R^2))$ for all $ 1 \leq p,q < \infty$ by Proposition~\ref{prop: td-displacements-further} so that         
        \begin{align*}
         (A_n \tilde J^n)_{\cdot l} \cdot  z_3^l \begin{pmatrix} \nabla' \phi(t,\bar{x}') \\ 0 \end{pmatrix} \underset{n \to \infty}{\longrightarrow} \tilde J_{\cdot l} \cdot z_3^l A^T \begin{pmatrix} \nabla' \phi(t,x') \\ 0 \end{pmatrix} 
			= 0
        \end{align*} 
        in $L^1_\mathrm{loc}(I_T; L^1(\Ox))$ since $\tilde J \perp \left( \R^3 \otimes e_3 \right)Z$. We may thus replace $R_n$ by $\Id$ in \eqref{term1} and are led to consider the limiting behavior of 
        \begin{equation}
			\int_0^{T'} \int_\Ox h_n^{-1} \tilde J^n \colon \bar{\nabla}_n \bar \varphi_1 \ dx \ dt 
			-\sum_{l=1}^{8} \int_0^{T'} \int_\Ox h_n^{-1} \tilde J^n_{\cdot l} \cdot  z_3^l \matr{(\nabla' \phi)\dtilde{\phantom{\iota}}(t,\bar{x}) \\ 0} \ dx \ dt.  \label{ELE1_td_15}
		\end{equation}
		Taylor expanding the discrete derivative $\bar{\nabla}_n \bar \varphi_1 = (\bar{\partial}_l^n \bar \varphi_1)_l$ and the interpolated gradient $(\nabla' \phi)\dtilde{\phantom{\iota}}$ as 
		\begin{align*}
            \bar{\partial}_l^n \bar \varphi_1(t, \bar{x}) 
            &= \Big( \nabla' \phi(t, \bar{x}) (z^l)' + \frac{\varepsilon_n}{2} \nabla'^2 \phi (t, \bar{x})\big[(z^l)',(z^l)'\big] \\ 
            &\qquad - \frac{1}{8} \sum_{j=1}^8\frac{\varepsilon_n}{2} \nabla'^2 \phi (t, \bar{x})\big[(z^j)',(z^j)'\big] + O(\varepsilon_n^2) \Big) e_3 \\ 
        \shortintertext{and } 
            (\nabla' \phi)\dtilde{\phantom{\iota}}(t,\bar{x}) 
            &= \frac{1}{8}\sum_{j=1}^8\nabla' \phi(t,\bar{x}' + (z^j)') 
            = \nabla' \phi(t,\bar{x}) + O(\varepsilon_n), 
        \end{align*}
        and noting that 
 		\[ \nabla'^2 \phi(t, x')\big[ (z^i)', (z^i)' \big] = \frac{1}{4} \big(\partial_{11}\phi(t,x') + (-1)^{i+1} \cdot 2\, \partial_{12}\phi(t,x') + \partial_{22}\phi(t,x') \big), \]
        we get
		\begin{align}
			\eqref{ELE1_td_15} 
			&= \sum_{l=1}^{8} \sum_{i=1}^{2} \int_{0}^{T'} \int_\Ox h_n^{-1} \Big[\tilde J^n_{3l} \partial_i \phi(t,\bar{x}')z^l_i 
			- \tilde J^n_{il}\partial_i \phi(t,\bar{x}') z_3^l \Big] \ dx \ dt \label{conv3-p6} \\
			&\qquad + \sum_{l=1}^{8} \int_{0}^{T'} \int_\Ox \frac{1}{2(\nu_n - 1)} \tilde J_{\cdot l}^n \cdot \frac{1}{2} \partial_{12} \phi(t,\tilde{x}') (-1)^{l+1} e_3 \ dx \ dt + O(\varepsilon_n). \label{conv3-p4}
		\end{align}
        Here the first term converges to $0$ as $n \to \infty$ by Lemma~\ref{lem:symmetric-properties-time-dependent}. 
       In case $\nu_n \to \infty$ the term \eqref{conv3-p4} clearly vanishes. If $\nu_n \equiv \nu \in \N$ however by part (b) of Corollary~\ref{cor:summary-tilde-J} we see that
		\begin{align*}
			\eqref{conv3-p4} 
			&\underset{n \to \infty}{\longrightarrow}
			\int_{0}^{T'} \int_S \frac{1}{4(\nu - 1)} DQ_{\cell}^{\rel} \Big(\hat{G}_1^{\sym} Z + \frac{1}{2(\nu-1)} ( \hat{G}_2 Z_- +\partial_{12} v \, M) \Big) \colon \partial_{12} \phi M \ dx' \ dt \\
			&\qquad + \int_{0}^{T'} \int_S \frac{1}{2(\nu - 1)^2} DQ_{\surf} \Big( \hat{G}_1^{\sym} Z^{(1)} + \frac{1}{2(\nu-1)} \partial_{12} v \, M^{(1)} \Big) \colon \partial_{12} \phi M^{(1)} \ dx' \ dt.
		\end{align*}
		This concludes {\bf Part 3}.
		
		Summarizing we see that, for every $\phi$ of the form $\phi(t,x') = \eta(t) \chi(x')$ with $\eta \in C_c^\infty((0,T'))$, $\chi \in C_c^\infty(S)$, \eqref{dyn_weakform1} for $\nu_n \to \infty$ follows from \eqref{timedep_ELE1_1}, \eqref{timedep_ELE1_5} and from \eqref{term1} $\to 0$ if $\nu_n \to \infty$, while \eqref{dyn_weakform1} for $\nu_n \equiv \nu$ follows from \eqref{timedep_ELE1_3}, \eqref{timedep_ELE1_5-nu} and \eqref{endpart3} if $\nu_n \equiv \nu \in \N$.
		In both cases by density the equation holds true for every $\phi \in L^2(0,T;H_0^2(S)) \cap H_0^1(0,T;L^2(S)).$
		
		\bigskip
		\textit{Step 5: Weak continuity and the initial conditions.}
		
		From inequality \eqref{energy_ineq2} it follows that, up to a subsequence,
		\[ \frac{1}{h_n^2} \int_0^1 \big({\tilde{y}}_n^{(1)} (\cdot,  x_3)\big)_3 \ dx_3 \rightharpoonup y_3^{(1)} \quad \text{in} \ L^2( \R^2) \]
		for some $y_3^{(1)} \in L^2(S)$ supported on $S$. Note that, since $\big(y_n^{(0)}(x)\big)_3 = h_n(x_3 - \frac{1}{2})$ on $\tilde{\Lambra}_n \setminus \tilde{\Lambda}_{n,S}$, by \cite[Lemma 13]{LecumberryMueller_09} and \eqref{energy_ineq2}
		\[ \frac{1}{h_n} \int_0^1 \big({\tilde{y}}_n^{(0)} (\cdot, x_3)\big)_3 \ dx_3 \to y_3^{(0)} \quad \text{in} \ H^1( \R^2) \]
		for some $y_3^{(0)} \in H^1(S)$ supported on $S$. Since $v_n, v \in W^{1,\infty}(0,T';L^2(\R^2)) \hookrightarrow C([0,T'];L^2( \R^2))$ we get for almost every $x' \in  \R^2$
		\begin{align*}
			y_3^{(0)}(x') 
			&= \lim_{n \to \infty} h_n^{-1} \int_0^1 \big({\tilde{y}}_n^{(0)}(x', x_3) \big)_3\ dx_3 \\ 
			&= \lim_{n \to \infty} h_n^{-1} \int_0^1 \big({\tilde{y}}_n(0,x', x_3) \big)_3 \ dx_3 
			= v(0,x')
		\end{align*}
		which is \eqref{init_cond_limit_1}. 
		In order to derive the initial condition \eqref{init_cond_limit_2} we let $\phi \in C_c^\infty((0,T')\times S )$. The goal is to obtain an estimate
		\begin{equation}
			\bigg|\int_0^{T'} \int_{\R^2} \partial_t v^n \, \partial_t \phi \ dx' \ dt - \int_0^{T'} \int_{\R^2} \sum_{\alpha=1}^{2} \partial_t q^n_\alpha \partial_t \partial_\alpha \phi \ dx' \ dt\bigg| \leq C \norm{\phi}_{L^2(0,T';H^5(S))}. \label{H-5_bound}
		\end{equation}
        To this end, we consider the test functions
		\begin{align*}
			\varphi_1(t,x) = (0,0,\phi(t,x')), 
			\qquad 
			\varphi_2(t,x) = \Big(x_3 - \frac{1}{2} \Big) \matr{\nabla' \phi(t,x') \\ 0}.
		\end{align*}
		Set 
		\begin{align*}
		v^n(t,x') 
		&= h_n^{-1} \int_{-\frac{1}{2(\nu_n - 1)}}^{\frac{2\nu_n - 1}{2(\nu_n - 1)}} (\bar y_n)_3 (t,x',x_3) \ d x_3 \\
		\shortintertext{and} 
		q^n(t,x') 
		&= \int_{-\frac{1}{2(\nu_n - 1)}}^{\frac{2\nu_n - 1}{2(\nu_n - 1)}} {\overline{\left(x_3 - \frac{1}{2}\right)}} \bar y_n^\prime(t,x) \ dx_3. 
		\end{align*}
		Computing the difference of the weak form of the equations of motions \eqref{eq:ODE-with-J} for $y_n$ with $h_n^{-1} \varphi_1$ and with $\varphi_2$ we get, writing again $R_n = Id + h_n A_n$, 
		\begin{align}
		\MoveEqLeft
			\int_0^{T'} \int_{\R^2} \partial_t  v^n \, \partial_t \bar \phi \ dx' \ dt 
			- \int_0^{T'} \int_{\R^2} \sum_{\alpha=1}^{2} \partial_t q^n_\alpha \, \partial_t \overline{\partial_\alpha \phi} \ dx' \ dt \label{phi-with-double-bars} \\
			&= \int_0^{T'} \int_\Ox A_n \tilde J^n \colon \bar{\nabla}_n \bar \varphi_1 \ dx \ dt 
			-\frac{\nu_n}{\nu_n - 1} \int_0^{T'} \int_{\R^2} \bar g \, \bar \phi \ dx' \ dt \label{rather-easy-terms} \\
			&\qquad + \int_0^{T'} \int_\Ox h_n^{-1} \tilde J^n \colon \bar{\nabla}_n \bar \varphi_1 \ dx \ dt 
			- \int_0^{T'} \int_\Ox R_n \tilde J^n \colon \bar{\nabla}_n \bar{\varphi}_2 \ dx \ dt. \label{terms-to-be-bounded}
		\end{align}
		Instead of \eqref{H-5_bound} we will first derive a corresponding estimate for \eqref{phi-with-double-bars}. 
        First we look at the first term in \eqref{rather-easy-terms} which contains the mapping $A_n$. We have
		\begin{align}
			\bigg|\int_0^{T'} \int_\Ox A_n \tilde J^n \colon \bar{\nabla}_n \bar \varphi_1 \ dx \ dt\bigg| 
			&\leq C \|A_n \tilde J^n\|_{L^2(0,T';L^1(\Ox))} \|\nabla' \phi\|_{L^2(0,T';L^\infty(\R^2))} \nonumber \\
			&\leq C \|\phi\|_{L^2(0,T'; H^5({\R^2))}} \label{initvalues-t1}
		\end{align}
		by Sobolev embeddding, Proposition~\ref{prop: td-displacements-further} and \eqref{eq: cor: scaled-strain-td-2}. Also the force term in \eqref{rather-easy-terms} is clearly bounded by 
		\begin{equation}
			\biggl \vert \int_0^{T'} \int_{\R^2} \bar g \bar \phi \ dx' \ dt \biggr \vert
			\leq C \norm{\phi}_{L^2(0,T';H^5(S))} \label{initvalues-t2}
		\end{equation}
		due to $g \in L^2(I_T; W^{1,\infty}(\R^2)) \cap C^0(I_T, L^\infty(\R^2))$. For the last term in \eqref{terms-to-be-bounded} we use \eqref{eq: prodrule-for-phi2} to expand the discrete gradient of $\bar\varphi_2$. Since 
		\begin{align*}
		\MoveEqLeft
		\Bigg| \sum_{l=1}^{8} \int_0^{T'} \int_\Ox \big(R_n \tilde J^n\big)_{\cdot l} \cdot  \Big( \bar{x}_3 + \frac{\varepsilon_n}{h_n}z_3^l - \frac{1}{2} \Big) \matr{ \bar{\partial}_l^n \bar\eta(t, \bar{x}') \\ 0 } \ dx \ dt \Bigg| \\ 
		&\leq C \|\tilde J^n\|_{L^2(0,T';L^1(\Ox))} \|\nabla'^2 \phi\|_{L^2(0,T';L^\infty(\R^2))} 
		\leq C \norm{\phi}_{L^2(0,T'; H^5(\R^2))}, 
		\end{align*}
		where $\eta_i(t,x') = \partial_i \phi(t,x')$, $i = 1,2$, it remains to bound
		\begin{equation}
			\biggl \vert \int_0^{T'} \int_\Ox h_n^{-1} \tilde J^n \colon \bar{\nabla}_n \doublebar \varphi_1 \ dx \ dt 
			- \sum_{l=1}^{8} \int_0^{T'} \int_\Ox \big(R_n \tilde J^n\big)_{\cdot l} \cdot h_n^{-1} z_3^l \matr{ \dtilde{\eta}(t,\bar{x}') \\ 0 } \ dx \ dt \biggr \vert.  \label{td_init_cond_7}
		\end{equation}
		This is exactly the term in \eqref{ELE1_td_15} above. There, in order to obtain \eqref{conv3-p6} and \eqref{conv3-p4}, we showed that 
		\begin{align}
            \MoveEqLeft
            h_n^{-1} \tilde J^n \colon \bar{\nabla}_n \bar \varphi_1 
            - \sum_{l=1}^{8} \big(R_n \tilde J^n\big)_{\cdot l} \cdot h_n^{-1} z_3^l \matr{ \dtilde{\eta}(t,\bar{x}' \\ 0)} \nonumber \\ 
			&= \sum_{i=1}^{2} \bigg( \sum_{l=1}^{8} h_n^{-1} \big[\tilde J^n_{3l} z^l_i 
			- \tilde J^n_{il} z_3^l \big] \bigg) \partial_i \phi(t,\bar{x}') \label{Jphi-commu-est} \ \\
			&\qquad + \sum_{l=1}^{8} \int_{0}^{T'} \int_\Ox \frac{1}{2(\nu_n - 1)} \tilde J_{\cdot l}^n \cdot \frac{1}{2} \partial_{12} \phi(t,\tilde{x}') (-1)^{l+1} e_3 \ dx \ dt + O(\varepsilon_n) \label{JD2phiOeps}
		\end{align}
		with an error term $O(\varepsilon_n)$ that is uniformly bounded by $C \eps_n \| \phi \|_{L^\infty(0,T';W^{3,\infty}(\R^2))}$. $\tilde J^n$ is bounded in $L^\infty(0,T'; L^2(\Ox))$, we find that 		
		\[ 
		 |\eqref{JD2phiOeps}|
		 \le C \| \phi \|_{L^\infty(0,T';H^5(\R^2))} 
		\]
        Finally, the term in \eqref{Jphi-commu-est} is bounded with the help of Lemma~\ref{lem:symmetric-properties-time-dependent} by
        \[ 
		 |\eqref{Jphi-commu-est}|
		 \le h_n \| \phi \|_{L^\infty(0,T';H^5(\R^2))}. 
		\]
		Together with \eqref{initvalues-t1} and \eqref{initvalues-t2} this yields the desired bound for \eqref{phi-with-double-bars}: 
		\begin{align}
		 \bigg|\int_0^{T'} \int_S \partial_t  v^n \, \partial_t \bar \phi \ dx' \ dt 
			- \int_0^{T'} \int_S \sum_{\alpha=1}^{2} \partial_t q^n_\alpha \, \partial_t \overline{\partial_\alpha \phi} \ dx' \ dt \bigg| 
		 \le C \| \phi \|_{L^\infty(0,T';H^5(\R^2))}. \label{H-5_bound-bars}
		\end{align}
        In order to prove \eqref{H-5_bound}, we apply \eqref{H-5_bound-bars} to the function 
        \[ \hat\phi = \varepsilon_n^{-2} \chi_{(-\frac{\varepsilon}{2},\frac{\varepsilon}{2})^2} \ast \phi \in C_c^\infty((0,T')\times S ) \]
        and note that, since $v^n$ and $q_\alpha^n$ are constant on each $x + (-\frac{\varepsilon}{2},\frac{\varepsilon}{2})^2$, $x \in \varepsilon \Z^2$, 
        \[\int_{\R^2} \partial_t  v^n \, \partial_t \phi \ dx' - \int_{\R^2} \sum_{\alpha=1}^{2} \partial_t q^n_\alpha \, \partial_t \partial_\alpha \phi \ dx' 
        = \int_{ \R^2} \partial_t  v^n \, \partial_t \bar{\hat{\phi}} \ dx' - \int_{\R^2} \sum_{\alpha=1}^{2} \partial_t q^n_\alpha \, \partial_t \overline{\partial_\alpha \hat{\phi}} \ dx'
        \] 
        and that 
        \[ 
         \| \hat\phi \|_{L^\infty(0,T';H^5(\R^2))} 
         \le \| \phi \|_{L^\infty(0,T';H^5(\R^2))}
        \]
        by Young's inequality. 

        To prove \eqref{init_cond_limit_2}, we first note that both sides vanish for a.e.\ $x' \in \R^2 \setminus S$. Now \eqref{H-5_bound} implies that the sequence $\partial_t^2 v^n + \sum_{\alpha=1}^{2} \partial_t^2 \partial_\alpha q^n_\alpha$ in bounded in $L^2(0,T'; H^{-5}(S))$. The partial derivatives $\partial_\alpha q_\alpha^n$, $\alpha= 1,2$, are understood as distributional derivatives here.	Since for the distributional derivative $\partial^\alpha f$ of an $L^2(S)$-function $f$ we always have the estimate $\norm{\partial_\alpha f}_{H^{-1}(S)} \leq \norm{f}_{L^2(S)}$ from \eqref{energy2} it follows that
		\[ \partial_t \partial_\alpha q^n_\alpha \to 0 \quad \text{in} \ L^\infty(0,T'; H^{-1}(S)). \]
		By Lemma~\ref{lemma: weak-*-interpolations} and  \eqref{eq: vn-convergence} it holds that
		\[  \partial_t  v^n \overset{*}{\rightharpoonup} 
		\begin{cases}
			\partial_t v & \text{if} \ \nu_n \to \infty, \\
			\frac{\nu}{\nu - 1} \partial_t v & \text{if} \ \nu_n \equiv \nu \in \N 
		\end{cases}
		\]
		in $L^\infty(0, T'; L^2(S))$ and therefore also in $L^\infty(0,T'; H^{-1}(S))$, hence we get
		\[ \partial_t v^n + \sum_{\alpha=1}^{2} \partial_t \partial_\alpha q^n_\alpha \overset{*}{\rightharpoonup} 
		\begin{cases}
			\partial_t v & \text{if} \ \nu_n \to \infty, \\
			\frac{\nu}{\nu - 1} \partial_t v & \text{if} \ \nu_n \equiv \nu \in \N 
		\end{cases}\]
		in $L^\infty(0,T'; H^{-1}(S))$. Since the embedding
		\[ L^\infty(0,T'; H^{-1}(S)) \cap H^1(0,T';H^{-5}(S)) \hookrightarrow C([0,T'];H^{-5}(S)) \]
		is compact (c.f. \cite{Simon1986}) we deduce
		\[ \frac{\nu_n - 1}{\nu_n} \left(\partial_t v^n(0,\cdot ) + \sum_{\alpha=1}^{2} \partial_t \partial_\alpha q^n_\alpha(0, \cdot) \right)\to \partial_t v(0, \cdot) \]
		strongly in $H^{-5}(S)$. The prefactor $\frac{\nu_n - 1}{\nu_n}$ is to avoid case distinction. Let $\phi \in C_c^\infty(S)$, then
		\begin{align*}
			\int_S y_3^{(1)}(x') \phi(x') \ dx'
			&= \lim_{n \to \infty} \frac{\nu_n - 1}{\nu_n} \int_S h_n^{-1} \big(\bar y_n^{(2)}\big)_3 \phi (x') \ dx' \\
			&= \lim_{n \to \infty} \frac{\nu_n - 1}{\nu_n}  \int_S \partial_t  v^n(0,x') \phi(x') \ dx' \\
			&= \lim_{n \to \infty} \frac{\nu_n - 1}{\nu_n}  \int_S \left(\partial_t  v^n(0,x') + \sum_{\alpha=1}^{2} \partial_t \partial_\alpha q^n_\alpha(0,x') \right) \phi(x') \ dx' \\
			&\qquad-\lim_{n \to \infty} \frac{\nu_n - 1}{\nu_n}  \int_S \sum_{\alpha=1}^{2} \partial_t \partial_\alpha q^n_\alpha(0,x') \phi(x') \ dx' \\
			&= \int_S \partial_t v(0,x') \phi(x') \ dx',
		\end{align*}
		and \eqref{init_cond_limit_2} follows.
		
		Finally we show the weak continuity of $I_T \to L^2(S)$, $t \mapsto \partial_t v$. Let $(t_n)_n \subset I_T$ and $t \in I_T$ such that $t_n \to t$. Since $\partial_t v \in L^\infty_{\mathrm{loc}}(I_T; L^2(S))$ the sequence $(\partial_t v(t_n))_{n \in \N}$ is bounded in $L^2(S)$ and therefore converges weakly to some $f \in L^2(S)$. But since $\partial_t v \in C([0,T']; H^{-5}(S))$ we have $\partial_t v(t_n) \to \partial_t v(t)$ in $H^{-5}(S)$. Hence $f = \partial_t v(t)$ and weak continuity follows.
		
		Similarly we can show weak continuity of the map $I_T \to H^2(S)$, $t \mapsto v(t)$. For every $T' \in I_T$ we have $v \in W^{1,\infty}(0,T'; L^2(S)) \hookrightarrow C([0,T']; L^2(S))$ and $v \in L^\infty(0,T'; H^2(S))$. Thus the sequence $(v(t_n))_n$ is bounded in $H^2(S)$, i.e. $v(t_n) \rightharpoonup f$ for some $f \in H_0^2(S)$. Further $v(t_n) \to v(t)$ in $L^2(S)$, hence $v = f$.
	\end{proof}

\section{The stationary case}
	Theorem~\ref{thm: main-td} implies a similar convergence result for the stationary version of the von Kármán equations. For a direct approach in the static regime we refer to \cite{Buchberger:24}.
	
	\begin{definition}
		Let $\nu \in \{2,3,\ldots,\infty\}$. We say a pair $(u,v) \in H_0^1(S;\R^2) \times H_0^2(S)$ is a distributional solution to the von Kármán equations if
		\begin{align}
			\int_S \left( 
			- \mathcal{B}^\nu_{{\rm vK},1}(u, v, \phi) 
			+ g \, \phi \right) \ dx' = 0 \label{stat_weakform1}
		\end{align}
		for every $\phi \in C_c^\infty(S)$, and
		\begin{align}
			\int_S \mathcal{B}^\nu_{{\rm vK},2}(u, v, \Psi)  \ dx' = 0 \label{stat_weakform2} 
		\end{align}
		for every $\Psi \in C_c^\infty(S; \R^2)$.
	\end{definition}

	\begin{theorem}\label{thm: stationary}
		Let $g \in W^{1,\infty}(S)$ and $y_n:\tilde{\Lambda}_{n,S} \to \R^3$ be a sequence of stationary points of $E_n$, i.e.,
        \begin{equation}
			0 = \frac{\varepsilon^3_n}{h_n} \left[ \sum_{x \in \tilde{\Lambda}^{\prime}_{n,S} } D_A W \left(x, \bar{\nabla}_n y_n(x) \right) \colon \bar{\nabla}_n \varphi(x)
			+ h_n^3 \sum_{x \in \tilde{\Lambda}_{n,S}} g(x') \cdot \varphi(x) \right] \label{stat-point}
		\end{equation}
        for every $\varphi: \tilde{\Lambda}_{n,S} \to \R^3$. Moreover $y_n$  should satisfy the energy bound \[E_n (y_n) \leq C h_n^4 \] as well as the fully clamped boundary conditions. Then, up to a subsequence, it holds that
		\begin{align}
			u_n(x') &= h_n^{-2} \int_{0}^{1} \left( \tilde{y}_n^\prime (x',x_3) - x' \right) \ dx_3 \rightharpoonup u \text{ in } H^1(S), \label{conv_un} \\
			v_n(x') &= h_n^{-1} \int_{0}^{1} \left( \tilde{y}_n \right)_3 \ dx_3 \to v \text{ in } H^1(S), \ v \in H^2(S). \label{conv_vn} 
		\end{align}
		The limiting pair $(u,v)$ is a solution to \eqref{stat_weakform1} and \eqref{stat_weakform2}.
	\end{theorem}
	
	\begin{proof}
		Equation \eqref{strong-solution} holds true time-independently. The mappings $y_n$ clearly satisfy the assumptions of Proposition~\ref{prop: td-displacements}. In particular the averaged scaled in- and out-of-plane displacements satisfy the convergences \eqref{weak-*-un}--\eqref{weak-*-dtvn} and fulfill the equations \eqref{dyn_weakform1} and \eqref{dyn_weakform2}. Equation \eqref{dyn_weakform1} holds true without the derivative term since $\partial_t v = 0$. Choosing test functions of the form $\phi(t,x) = \eta(t)\chi(x')$ with $\eta \in C_c^\infty(0,T')$, $\chi \in C_c^\infty(S)$, or $\Psi(t,x') = \eta(t) \tilde{\Psi}(x')$ with $\tilde{\Psi} \in C_c^\infty(S;\R^2)$, together with the fact that both, $u$ and $v$ are time-independent, implies \eqref{stat_weakform1} and \eqref{stat_weakform2}. The boundary conditions follow from Proposition~\ref{prop: td-displacements}.
	\end{proof}

\section*{Acknowledgments}
The second author gratefully acknowledges the support of the Deutsche Forschungsgemeinschaft (DFG, German Research Foundation) for Project 441138507.

\bibliographystyle{alpha} 
\bibliography{references}

\end{document}